\documentclass[11pt,a4paper,fleqn]{article}

\usepackage{amsmath,amssymb,amsthm,enumerate}

\newcommand{\ev}{\mathop{\mathbf{ev}}\nolimits}
\newcommand{\ann}{\mathop{\rm ann}\nolimits}
\newcommand{\rank}{\mathop{\rm rank}\nolimits}
\newcommand{\im}{\mathop{\rm im}\nolimits}
\newcommand{\Char}{\mathop{\rm char}\nolimits}
\newcommand{\ad}{\mathop{\rm ad}\nolimits}

\newcommand{\F}{\mathbb{F}}
\newcommand{\LnF}{\mathop{\mathcal L_n(\F)}}
\newcommand{\AnF}{\mathop{\mathcal A_n(\F)}}
\newcommand{\SknF}{\mathop{\mathcal{K}_n(\F)}}
\newcommand{\BnF}{\mathop{\mathcal{B}_n(\F)}}
\newcommand{\UnF}{\mathop{\mathcal{U}_n(\F)}}

\newcommand{\GlnF}{\mathop{{\rm GL}(n,\F)}\nolimits}
\newcommand{\Fspan}{\mathop{\F\mbox{-\rm sp}}\nolimits}

\newcommand{\LnFa}{\mathop{\mathcal L_n^a(\F)}}
\newcommand{\SknFa}{\mathop{\mathcal{K}_n^a(\F)}}
\newcommand{\BnFa}{\mathop{\mathcal{B}_n^a(\F)}}

\newcommand{\qadeg}{\mathop{\hat q\_\rm{aux\_deg}}\nolimits}

\newcounter{mcasenum}

\newtheorem{theorem}{Theorem}[section]
\newtheorem{lemma}[theorem]{Lemma}
\newtheorem*{lemma*}{Lemma}
\newtheorem{corollary}[theorem]{Corollary}
\newtheorem*{corollary*}{Corollary}

{\theoremstyle{definition} \newtheorem{definition}[theorem]{Definition}
\newtheorem*{definition*}{Definition}
\newtheorem{example}[theorem]{Example}
\newtheorem*{examples*}{Exampes}
\newtheorem{remark}[theorem]{Remark}
\newtheorem*{remark*}{Remark}

\newtheorem{result}[theorem]{Result}
\newtheorem*{result*}{Result}
\newtheorem*{comment*}{Comment}

\begin{document}

\title{On degenerations of algebras over an arbitrary field}

\author{N.M. Ivanova$^a$ and C.A. Pallikaros$^{b}$}

\date{February 13, 2017}
\maketitle

{\vspace{1mm}\par\noindent\footnotesize\it{
${}^{a}$~Institute of Mathematics of NAS of Ukraine,~3 Tereshchenkivska Str., 01601 Kyiv, Ukraine\\
\phantom{${}^{a}$~}{}E-mail: ivanova.nataliya@gmail.com\\
${}^{b}$~Department of Mathematics and Statistics, University of Cyprus, PO Box 20537,
1678 Nicosia, Cyprus\\
\phantom{${}^{b}$~}{}E-mail: pallikar@ucy.ac.cy\\

}}

\begin{abstract}
For each $n\ge2$ we classify all  $n$-dimensional algebras over an arbitrary infinite field which have the property that the $n$-dimensional abelian Lie algebra is their only proper degeneration.

\medskip\noindent
%
%
\end{abstract}


\section{Introduction}

The concept of degeneration probably first arises in the second half of the twentieth century when a lot of attention was paid to the study of various limit processes linking physical theories.
One of the most famous examples of such a limit process is the relation between classical and quantum mechanics.
Indeed, classical mechanics can be studied as a limit case of quantum mechanics where the quantum mechanical commutator $[x,p]=i\hslash$ (corresponding to the Heisenberg uncertainty principle) maps to the abelian case (that is the classical mechanics limit) as $\hslash\to 0$.
The pioneering works in this direction were a paper by Segal~\cite{Segal1951}
who considered the non-isomorphic limit of sequences of structure constants of isomorphic Lie groups, and a series of papers of In\"on\"u and Wigner~\cite{InonuWigner1953,InonuWigner1954} devoted to the limit process $c\to\infty$ in special relativity theory showing how the symmetry group of relativistic mechanics (the Poincar\'e group) degenerates to the symmetry group of classical mechanics (Galilean group).
The target algebras of such limit processes (which are nothing else but the points in the closure in metric topology of the orbit of the initial algebra under the `change of basis' action of the general linear group) are called contractions (or degenerations in the more general context of an arbitrary field and Zariski closure).
These (and many other) examples of degenerations have a wide range of applications in physics based on the fact that
if two physical theories are related by a limiting process, then the associated invariance groups (and invariance algebras) should also be related by some limiting process.
Degenerations of algebras also have applications in other branches of mathematics. Thus, for example, they can be used as a tool for finding rigid algebras which are important for the investigation of varieties of algebras and their irreducible components.
Degenerations are widely used for studying different properties of flat and curved spaces and in the theory of quantum groups.
As the notion of degeneration is closely related to the notion of deformation~\cite{Gerstenhaber1964}, degenerations are often used for investigating deformations of Lie algebras~\cite{Levy-Nahas1967,FialowskyOHalloran1990}.

Despite their theoretical and practical interest, results about degenerations, especially in fields other then~$\mathbb C$ or $\mathbb R$, are still fragmentary.
In~\cite{GrunewaldOHalloran1988a}, it is shown that the closures in the Zariski topology and in the standard topology of the orbit of a point of an affine variety over $\mathbb C$ under the action of an algebraic group coincide.
In the same work, again over $\mathbb C$, for $G$ a reductive algebraic group with Borel subgroup $B$ and $X$ an algebraic set on which $G$ acts, it is shown that for $x\in X$ the closure of the $B$-orbit of $x$ has non-empty intersection with every orbit in the closure of its $G$-orbit.
Some necessary conditions for the existence of degenerations of real and complex algebras were constructed  in~\cite{GrunewaldOHalloran1988a,NesterenkoPopovych2006,Rakhimov2005,Seeley1990,Seeley1991}.
A criterion for a Lie algebra to be a degeneration of another Lie algebra over an algebraically closed field is given in~\cite{GrunewaldOHalloran1988b}.
In~\cite{Popov2009} the question of whether or not a given orbit (of a point under the action of an algebraic group) lies in the Zariski closure of another orbit is being considered and a method of solving this problem is presented.
However, in practice it is extremely difficult to apply the results of~\cite{GrunewaldOHalloran1988b,Popov2009}.
For some classes of algebras (like real and complex 3- and 4-dimensional Lie algebras, low-dimensional nilpotent Lie algebras, some subclasses of Malcev algebras) the problem of determining all degenerations within the given class has been considered, see for example~\cite{Burde&Steinhoff1999,NesterenkoPopovych2006,Kaygorodov&Popov&Volkov2016,Kaygorodov&Volkov2017} and references therein.

Our motivation comes from works of Gorbatsevich~\cite{Gorbatsevich1991,Gorbatsevich1993,Gorbatsevich1998} and in particular the notion of the level of complexity of a finite dimensional algebra.

At this point we introduce some notation.
Let $V$ be an $n$-dimensional vector space over an arbitrary field $\F$ and let $G=\GlnF$.
Also denote by $\AnF$ the set of all algebra structures on $V$ and by $\SknFa$ the subset of $\AnF$ consisting of the algebra structures satisfying the identity $[x,x]=0$.
(We denote by $[x_1,x_2]$ the product of the elements $x_1,$ $x_2$ of an algebra.)
Next we introduce some particular algebras in $\AnF$.
Let $\mathfrak{a}_n$ be the $n$-dimensional abelian Lie algebra.
Also let $\mathfrak{h}_n\in\AnF$ be isomorphic to the Lie algebra direct sum $\mathfrak{h}_3\oplus\mathfrak{a}_{n-3}$ (where $\mathfrak{h}_3$ is the Heisenberg algebra) and $\mathfrak{r}_n\in\AnF$ be defined as the algebra structure for which the only nonzero products between the basis elements of a fixed basis $(v_1^*,\ldots,v_n^*)$ of $V$ are $[v_i^*,v_n^*]=v_i^*=-[v_n^*,v_i^*]$ for $1\le i\le n-1$.
In particular $\mathfrak{r}_n$, $\mathfrak{h}_n$ are both Lie algebras and hence they both belong to $\SknFa$.
Note that $\mathfrak{r}_n$ is the only algebra structure in $\SknFa$ other than $\mathfrak{a}_n$ which satisfies the condition that the product of any two of its elements is a linear combination of the same two elements.

We consider the natural `change of basis' action of $G$ on $\F^{n^3}$, regarded as a space of structure vectors.
(We say that $\boldsymbol\lambda\in\F^{n^3}$ is a structure vector for $\mathfrak{g}\in\AnF$ if there exists a basis of $V$ relative to which $\boldsymbol\lambda$ is the vector of structure constants for $\mathfrak{g}$.)
The algebra structure $\mathfrak{g}$ degenerates to $\mathfrak{g}_1$ ($\mathfrak{g}$, $\mathfrak{g}_1\in\AnF$) if there exist $\boldsymbol\lambda$, $\boldsymbol\mu$ structure vectors for $\mathfrak{g}$, $\mathfrak{g}_1$ respectively ($\boldsymbol\lambda$, $\boldsymbol\mu\in\F^{n^3}$) such that $\boldsymbol\mu$ belongs to the Zariski-closure of the orbit of $\boldsymbol\lambda$ resulting from the above action.

One can easily observe that if $\F$ is infinite then every $\mathfrak{g}\in\AnF$ degenerates to $\mathfrak{a}_n$.
Our first result deals with the problem of determining the isomorphism classes of algebras in $\SknFa$ that have $\mathfrak{a}_n$ as their only proper degeneration in the case $\F$ is an arbitrary infinite field.
Note that the special case $\F=\mathbb C$ is already settled in~\cite[Theorem~1]{Gorbatsevich1991}.
See also~\cite[Theorem~5.2]{Lauret2003} for the case $\F=\mathbb R$.

\begin{theorem}\label{TheoremMain}
Let $n\ge3$ and let $\F$ be an arbitrary infinite field.
Then, among all $n$-dimensional algebras satisfying the identity $[x,x]=0$, algebras $\mathfrak{r}_n$ and $\mathfrak{h}_n$ are the only ones (up to isomorphism) which have the $n$-dimensional abelian Lie algebra $\mathfrak{a}_n$ as their only proper degeneration.
\end{theorem}

Before stating the next result we need to introduce some more notation.
Let $\boldsymbol{\delta}_n=(\delta_{ijk})$ and $\boldsymbol{\varepsilon}_n(\alpha)=(\varepsilon_{ijk}(\alpha))$, for $\alpha\in\F$, be the elements of $\F^{n^3}$ which are defined by  $\delta_{112}=1_{\F}$ (all other $\delta_{ijk}$ being zero) and $\varepsilon_{111}(\alpha)=1_{\F}$, $\varepsilon_{1ii}(\alpha)=\alpha$, $\varepsilon_{i1i}(\alpha)=(1_{\F}-\alpha)$ for $2\le i\le n$
(all other $\varepsilon_{ijk}(\alpha)$ being zero).
Also, for $\alpha\in\F$, let $\mathfrak{d}_n$ and $\mathfrak{e}_n(\alpha)$ be the elements of $\AnF$ whose structure vectors relative to our fixed basis $(v_i^*)_{i=1}^n$ are $\boldsymbol{\delta}_n$ and $\boldsymbol{\varepsilon}_n(\alpha)$ respectively.

\begin{theorem}\label{TheoremMainPart2}
Let $n\ge3$ and let $\F$ be an arbitrary infinite field.
Then $\mathfrak{d}_n$ together with the family $\{\mathfrak{e}_n(\alpha):\ \alpha\in\F\}$ give a complete list of non-isomorphic elements of $\AnF\setminus\SknFa$ which have $\mathfrak{a}_n$ as their only proper degeneration.
\end{theorem}

We remark that the special case $\F=\mathbb C$ in Theorem~\ref{TheoremMainPart2} is already settled in~\cite[Theorem~2.3]{Khudoyberdiyev&Omirov2013}.

\medskip
The paper is organized as follows.
In Section~\ref{SectionPreliminaries} we give some necessary background.
In Section~\ref{SectionDegenerations} we introduce the notion of degeneration and discuss some necessary conditions for an algebra to degenerate to another.
Moreover, in Section~\ref{SectionDegenerations}, a key result is proved (Lemma~\ref{LemmaTechnique*}) which enables us to `translate' results of limiting processes (involving diagonal matrices) in the metric topology of $\mathbb C$ or $\mathbb R$ to results in the Zariski topology over an arbitrary infinite field.
In Section~\ref{SectionLevel1} we provide a completely elementary self-contained proof of Theorem~\ref{TheoremMain} (where we also use ideas from~\cite{Khudoyberdiyev&Omirov2013,Lauret2003}) and we discuss how the conclusion of this theorem can be used to obtain information about the composition series of $\SknFa$ (regarded as an $\F G$-module) via the action of $G=\GlnF$ we are considering.
Finally, in Section~\ref{SectionLevel1NonSkewSym} we provide a proof of Theorem~\ref{TheoremMainPart2} which uses ideas from~\cite{Khudoyberdiyev&Omirov2013}.


\section{Preliminaries}\label{SectionPreliminaries}
Fix a positive integer $n$ with $n\ge2$ and an arbitrary field $\F$.
Also let $G=\GlnF$.

It will be convenient to regard the $n^3$ triples $(i,j,k)$ for $1\le i,j,k\le n$ as an ordered $n^3$-tuple
of triples by placing, for $1\le m\le n^3$, at the $m^{\rm th}$ position of this $n^3$-tuple the triple $(i_1,j_1,k_1)$ where $i_1$, $j_1$, $k_1$ are the unique integers with $1\le i_1,j_1,k_1\le n$ satisfying $m-1=(i_1-1)n^2+(j_1-1)n+(k_1-1)$.

\begin{definition}\label{DefAlgebra}
(i) An $n$-dimensional $\F$-algebra  (not necessarily associative) is a pair $(A,[,])$ where $A$ is a vector space over $\F$ with $\dim_{\F}A=n$ and
$[,]:A\times A\to A:$ $(x,y)\mapsto[x,y]$ ($x,y\in A$) is an $\F$-bilinear map.
We call $[x,y]$ the product of $x$ and $y$.

(ii) An ideal of an $\F$-algebra $(A,[,])$ is an $\F$-subspace $J$ of $A$ such that $[a,b]$, $[b,a]\in J$ for all $b\in J$ and $a\in A$.
\end{definition}

We comment here that it is more common to use the notation $xy$ for the product of $x$ and $y$.
In this paper, however, we adopt the bracket notation as our main interest is focused on Lie algebras.

Let $(A,[,])$ be an $n$-dimensional $\F$-algebra and suppose that $(b_1,\ldots,b_n)\in A^n$ is an ordered $\F$-basis for the vector space $A$.
Then, by bilinearity of the bracket, we see that the multiplication in $(A, [,])$ is completely determined by the products $[b_i,b_j]$, $1\le i,j\le n$.
The structure constants of this algebra with respect to the basis $(b_1,\ldots,b_n)$ are the scalars $\alpha_{ijk}\in\F$ ($1\le i,j,k\le n$) given by
$[b_i,b_j]=\sum_{k=1}^n\alpha_{ijk}b_k$.
We will regard this set of structure constants $\alpha_{ijk}$ as an ordered $n^3$-tuple $(\alpha_{ijk})_{1\le i,j,k\le n}$ in $\F^{n^3}$ via the ordering on the triples $(i,j,k)$ for $1\le i,j,k\le n$ we have fixed above.
(For example, when $n=2$, we have $\alpha_{ijk}=(\alpha_{111},\alpha_{112},\alpha_{121},\alpha_{122},\alpha_{211},\alpha_{222})\in\F^{8}$.)
We call the $n^3$-tuple $\boldsymbol\alpha=(\alpha_{ijk})\in\F^{n^3}$ the structure vector of $(A, [,])$ relative to the $\F$-basis $(b_1,\ldots,b_n)$ of~$A$.
More generally, we call the element $\boldsymbol\beta=(\beta_{ijk})_{1\le i,j,k\le n}\in\F^{n^3}$ a structure vector for $(A, [,])$ if there exists an ordered $\F$-basis $(b_1',\ldots,b_n')$ for $A$ relative to which the structure vector for $(A,[,])$ is $\boldsymbol\beta=(\beta_{ijk})$,
that is $[b_i',b_j']=\sum_{k=1}^n\beta_{ijk}b_k'$ for $1\le i,j\le n$.

\medskip
Recall that the $n$-dimensional $\F$-algebras $(A_1,[,]_1)$ and $(A_2,[,]_2)$ are called isomorphic if there exists a bijective $\F$-linear map $\Psi:A_1\to A_2$ such that $\Psi([x,y]_1)=[\Psi(x),\Psi(y)]_2$ for all $x,y\in A_1$.

\begin{remark}\label{RemarkIsomorphicAlgStructureVectors}
(i) Suppose that $\Psi:A_1\to A_2$ defines an isomorphism between the $n$-dimensional $\F$-algebras $(A_1,[,]_1)$ and $(A_2,[,]_2)$.
Also let $(c_1,\ldots,c_n)$ be an $\F$-basis of $A_1$.
It is easy to observe that the structure vectors of $(A_1,[,]_1)$ and $(A_2,[,]_2)$ relative to the $\F$-bases
$(c_1,\ldots,c_n)$ and $(\Psi(c_1),\ldots,\Psi(c_n))$ of $A_1$ and $A_2$ respectively coincide.

(ii) Conversely, if the two $n$-dimensional $\F$-algebras $(A_1,[,]_1)$ and $(A_2,[,]_2)$ have a common structure vector, then they are necessarily $\F$-isomorphic.
\end{remark}

\smallskip
For the rest of the paper we fix $V$ to be an $n$-dimensional $\F$-vector space.
We also fix $(v_1^*,\ldots,v_n^*)$ to be an ordered $\F$-basis of $V$.
We will be referring to this basis of $V$ on several occasions in the sequel.

\begin{definition}\label{DefAlgebraStruct}
We call $\mathfrak{g}$ an algebra structure on~$V$ if $\mathfrak{g}$ is an $\F$-algebra having $V$ as its underlying vector space
(and hence has multiplication defined via a suitable $\F$-bilinear map $[,]_{\mathfrak{g}}:V\times V\to V$).
We denote by~$\AnF$ the set of all algebra structures on~$V$.
\end{definition}

It is then clear from the preceding discussion that any $n$-dimensional $\F$-algebra is isomorphic to an element of $\AnF$.
For $\mathfrak{g}_1,\mathfrak{g}_2\in\AnF$ we write $\mathfrak{g}_1\cong \mathfrak{g}_2$ to denote the fact that $\mathfrak{g}_1$ is isomorphic to
$\mathfrak{g}_2$ as $\F$-algebras.

Also observe that if $\mathfrak{g}=(V,[,]_{\mathfrak{g}})$ is an algebra structure on~$V$, we have $[0_V,x]_{\mathfrak{g}}=0_V=[x,0_V]_{\mathfrak{g}}$ for all $x\in V$.
Hence, the zero vector $0_V$ of $V$ is the zero element of any $\mathfrak{g}\in\AnF$.

\begin{remark}\label{RemarkAnFFn3AreVectorSpaces}
We can regard $\AnF$ as an $\F$-vector space:
For $\mathfrak{g}_1=(V,[,]_1)$, $\mathfrak{g}_2=(V,[,]_2)\in\AnF$ and $\alpha\in\F$ we define $\mathfrak{g}_1+\mathfrak{g}_2=(V,[,])\in\AnF$ where
$[x,y]=[x,y]_1+[x,y]_2$ and $\alpha \mathfrak{g}_1=(V,[,]_\alpha)$ where $[x,y]_\alpha=\alpha[x,y]_1$ for all $x,y\in V$.
It is then easy to check that the maps $[,]$, $[,]_\alpha:V\times V\to V$ are both $\F$-bilinear and that they turn $\AnF$ into an $\F$-space.

Regarding $\F^{n^3}$ as an $\F$-vector space, as usual, via the natural (componentwise) addition and scalar multiplication,
we can then obtain an isomorphism of $\F$-vector spaces $\Theta:\AnF\to\F^{n^3}$ such that the image of an algebra structure $\mathfrak{g}\in\AnF$ is the structure vector of $\mathfrak{g}$ relative to the basis $(v_1^*,\ldots,v_n^*)$ of $V$ we have fixed.
\end{remark}

\begin{definition}\label{DefOmegaMap}
With the help of the bijection $\Theta:\AnF\to\F^{n^3}$, given in Remark~\ref{RemarkAnFFn3AreVectorSpaces} we can define a map $\Omega:\F^{n^3}\times G\to\F^{n^3}:$ $(\boldsymbol\lambda,g)\mapsto\boldsymbol\lambda g$ ($\boldsymbol\lambda\in\F^{n^3}$, $g=(g_{ij})\in G$) where
$\boldsymbol\lambda g\in\F^{n^3}$ is the structure vector of $\Theta^{{}^{-1}}(\boldsymbol\lambda)\in\AnF$
relative to the basis $(v_1,\ldots,v_n)$ of $V$ given by $v_j=\sum_{i=1}^ng_{ij}v_i^*$.
(We call $g\in G$ the transition matrix from the basis $(v_i^*)_{i=1}^n$ to the basis $(v_i)_{i=1}^n$ of $V$.)
\end{definition}

\begin{remark}\label{RemarkChangeOfBasisStructureVector}
With the setup and notation of Definition~\ref{DefOmegaMap}, we see that $\Theta^{{}^{-1}}(\boldsymbol\lambda)$ is the algebra structure $\mathfrak{g}_1$ on $V$ having $\boldsymbol\lambda$ as its structure vector relative to the basis $(v_1^*,\ldots,v_n^*)$ of $V$ we have fixed.
If now $\mathfrak{g}_2\in\AnF$ also has structure vector $\boldsymbol\lambda$ relative to some basis $(u_i)_{i=1}^n$ of $V$, then the structure vector of $\mathfrak{g}_2$ relative to the basis $(u_i')_{i=1}^n$ of $V$ given by $u_j'=\sum_{i=1}^ng_{ij}u_i$ is again $\boldsymbol\lambda g$.
This means that the resulting vector $\boldsymbol\lambda g\in\F^{n^3}$ obtained from $\boldsymbol\lambda\in\F^{n^3}$ via the `change of basis' process in Definition~\ref{DefOmegaMap} which is determined by the element $g=(g_{ij})\in G$ does not depend on the choice of the pair
$(\mathfrak{g}\in\AnF,$ ordered basis for $V)$ that determines $\boldsymbol\lambda$ (note that in general there are more than one such pairs).
\end{remark}

It is then easy to observe that the map $\Omega$ defines a linear right action of $G$ on $\F^{n^3}$.
We can thus regard $\F^{n^3}$ as a right $\F G$-module via this action.

\medskip
The orbit of $\boldsymbol\lambda\in\F^{n^3}$ with respect to the action given in Definition~\ref{DefOmegaMap} will be denoted by $O(\boldsymbol\lambda)$.

\medskip

Note that the orbits resulting from this action correspond precisely to isomorphism classes of $n$-dimensional $\F$-algebras
(this is immediate from Remark~\ref{RemarkIsomorphicAlgStructureVectors}).

We now recall briefly some basic facts on algebraic sets.

Let $\F[{\bf X}]$ be the ring $\F[X_{ijk}:\ 1\le i,j,k\le n]$ of polynomials in the indeterminates $X_{ijk}$ ($1\le i,j,k\le n$) over $\F$.
For each $\boldsymbol\mu=(\mu_{ijk})\in\F^{n^3}$ we can define the evaluation map $\ev_{\boldsymbol\mu}:\F[{\bf X}]\to \F$
to be the unique ring homomorphism $\F[{\bf X}]\to \F$ such that $X_{ijk}\mapsto \mu_{ijk}$ for $1\le i,j,k\le n$
and which is the identity on $\F$.
A subset $W$ of $\F^{n^3}$ is algebraic (and thus closed in the Zariski topology on $\F^{n^3}$) if there exists a subset
$S\subseteq\F[{\bf X}]$ such that
$W=\{\boldsymbol\mu=(\mu_{ijk})\in\F^{n^3}:$ $\ev_{\boldsymbol\mu}(f)=0_{\F}$ for all $f\in S\}$.
The Zariski closure of a subset $Y$ of $\F^{n^3}$ will be denoted by $\overline{Y}$.

\medskip
Next we introduce some subsets of $\AnF$ which are defined via identities.

\vspace{-2ex}
\begin{enumerate}[(\it i\rm)]
\item
$\BnFa=\{\mathfrak{g}=(V,[,])\in\AnF:$ $[[x_1,x_2],x_3]=0_V$ for all $x_1,x_2,x_3\in V\}$.\\
(This is a set of metabelian algebra structures on $V$.)

\item
$\SknFa=\{\mathfrak{g}=(V,[,])\in\AnF:$ $[x,x]=0_V$ for all $x\in V\}$.\\
(Note that $[x,x]=0_V$ for all $x\in V$ implies that $[x,y]=-[y,x]$ for all $x,y\in V$.
In particular, for $\Char\F\ne2$, we have that $\SknFa$ is precisely the set of skew-symmetric algebra structures on $V$.)

\item
$\LnFa=\{\mathfrak{g}=(V,[,])\in\SknFa:$ $[x,[y,z]]+[y,[z,x]]+[z,[x,y]]=0_V$ for all $x,y,z\in V\}$.\\
(This is the set of algebra structures in $\SknFa$ satisfying the Jacobi identity and hence it is precisely the set of Lie algebra structures on $V$.)
\end{enumerate}

Denote by $\BnF$, $\SknF$, $\LnF$ the subsets of $\F^{n^3}$ which are the images of $\BnFa$, $\SknFa$, $\LnFa$ respectively via the map $\Theta$ given in Remark~\ref{RemarkAnFFn3AreVectorSpaces}.

\begin{remark}\label{RemarkZarClosedSets}
It is immediate from the way they are defined that $\BnF$, $\SknF$ and $\LnF$ are all unions of orbits.
Below we will also see that these subsets of $\F^{n^3}$ are algebraic (Zariski-closed) as they can be described via polynomial equations.
For the rest of the discussion in this remark $(e_1,\ldots,e_n)$ denotes an arbitrary $\F$-basis of $V$ and $\mathfrak{g}=(V,[,])$ an element of $\AnF$.

(i)
To establish that $\BnF$ is algebraic, first observe that $\mathfrak{g}$ is metabelian if, and only if, $[[e_i,e_j],e_k]=0_V$ for $1\le i,j,k\le n$.
Consequently, we have $\BnF=\{\boldsymbol\lambda=(\lambda_{ijk})\in\F^{n^3}:\ \sum_l\lambda_{ijl}\lambda_{lkm}=0_{\F}$ for $1\le i,j,k,m\le n\}$
showing that $\BnF$ is a Zariski-closed subset of $\F^{n^3}$.

(ii)
Also observe that $\mathfrak{g}\in\SknFa$ if, and only if, we have $[e_i,e_i]=0_V$ and $[e_i,e_j]+[e_j,e_i]=0_V$ for $0\le i,j\le n$.
%
It follows that $\SknF=\{\boldsymbol\lambda=(\lambda_{ijk})\in\F^{n^3}:\ \lambda_{iik}=0_{\F}$ and $\lambda_{ijk}+\lambda_{jik}=0_F$ for all $1\le i,j,k\le n\}$.

Clearly $\SknF$  is an $\F$-subspace of $\F^{n^3}$ ($\dim_{\F}\SknF=\frac12n^2(n-1)$) and hence it is an algebraic set.
(Alternatively, it is easy to see that $\boldsymbol\mu\in\SknF$ if, and only if,
$\ev_{\boldsymbol\mu}(f)=0_{\F}$ for all $f\in\{X_{iik}:\ 1\le i,k\le n\}\cup\{X_{ijk}+X_{jik}:\  1\le i,j,k\le n\}\subseteq \F[{\bf X}]$).

(iii)
Similar argument shows that
$\LnF=\{\boldsymbol\lambda=(\lambda_{ijk})\in\SknF:\ \sum_k(\lambda_{ijk}\lambda_{klm}+\lambda_{jlk}\lambda_{kim}+\lambda_{lik}\lambda_{kjm})=0_{\F}$ for $1\le i,j,l,m\le n\}$.
(See also~\cite[page~4]{Jacobson1962}.)
\end{remark}

Since $\SknF$ is an $\F$-subspace of $\F^{n^3}$ which is also a union of orbits, it can be regarded as an $\F G$-module 
via the (linear) action of $G$ on $\F^{n^3}$ we are considering.

For the rest of the paper, by an $\F G$-submodule of $\F^{n^3}$ we will mean any subspace of $\F^{n^3}$ which is also an $\F G$-module via the action of $G$ on $\F^{n^3}$ given in Definition~\ref{DefOmegaMap}.

\begin{example}\label{ExampleAbelianOrbit}
Let $\mathfrak{a}_n$ be the algebra structure on $V$ with multiplication defined by $[x,y]=0_V$ for all $x,y\in V$.
Then the structure vector of $\mathfrak{a}_n$ relative to any basis of $V$ is the zero vector ${\bf 0}=(0_{\F},\ldots,0_{\F},\ldots,0_{\F})$ of $\F^{n^3}$.
The orbit of ${\bf 0}$ under the action of $G$ we have described thus consists of precisely one point and hence it is Zariski-closed.
Clearly $\mathfrak{a}_n\in\LnFa$, in other words it is a Lie algebra structure on $V$.
The algebra $\mathfrak{a}_n$ is known as the $n$-dimensional abelian Lie algebra over $\F$.

Observe that for $\Char \F\ne2$ the only algebra structure $\mathfrak{g}=(V,[,])\in\SknF$ satisfying the commutativity relation $[x,y]=[y,x]$ for all $x,y\in V$ is the abelian Lie algebra $\mathfrak{a}_n$.
Note, however, that when $\Char\F=2$, we have that every skew-symmetric algebra satisfies the commutativity relation.
\end{example}

\section{Degenerations}\label{SectionDegenerations}

Recall that for $\boldsymbol\lambda\in\F^{n^3}$ and $g\in G$ we denote by $\boldsymbol\lambda g$ the image of $(\boldsymbol\lambda,g)\in\F^{n^3}\times G$ under the map $\Omega$ in Definition~\ref{DefOmegaMap}.

For each $g\in G$ define a function $\Phi_g:\F^{n^3}\to\F^{n^3}:$ $\boldsymbol\mu\mapsto\boldsymbol\mu g$ ($\boldsymbol\mu\in\F^{n^3}$).
Then $\Phi_g$ is a regular map for each $g\in G$ and hence continuous in the Zariski topology.
[To see this, fix $g\in G$.
It follows from the change of basis process that for each $\boldsymbol\mu\in\LnF$ we get
$\varphi_g(\boldsymbol\mu)=(\ev_{\boldsymbol\mu}(f_1),\ldots,\ev_{\boldsymbol\mu}(f_{n^3}))$ where, for $1\le i\le n^3$, $f_i$ is a polynomial
in $\F[\boldsymbol X]$ which only depends on the element $g\in G$.]

\medskip
The following result is elementary but for completeness we provide a proof for it.

\vspace{-2ex}
\begin{lemma}\label{LemmaUnionOfOrbitsTransitivity}
Let $\boldsymbol\lambda,\boldsymbol\mu\in\F^{n^3}$ with $\mu\in\overline{O(\boldsymbol\lambda)}$.
Then $O(\boldsymbol\mu)\subseteq \overline{O(\boldsymbol\lambda)}$ (and hence $\overline{O(\boldsymbol\mu)}\subseteq \overline{O(\boldsymbol\lambda)}$).

\end{lemma}

\begin{proof}
Assume the hypothesis and let $\mu^*\in O(\boldsymbol\mu)$.
It follows that $\boldsymbol\mu^*=\boldsymbol\mu g^*$ for some $g^*\in G$.
Suppose now that $U$ is an open subset of $\F^{n^3}$ containing $\boldsymbol\mu^*$.
It is enough to show that $O(\boldsymbol\lambda)\cap U\ne\varnothing$.
Invoking the fact that the map $\Phi_{g^*}:\F^{n^3}\to\F^{n^3}$ is continuous, we see that $\Phi_{g^*}^{-1}(U)$ is an open subset of $\F^{n^3}$.
Now $\boldsymbol\mu^*=\boldsymbol\mu g^*=\Phi_{g^*}(\boldsymbol\mu)\in U$, so $\boldsymbol\mu\in\Phi_{g^*}^{-1}(U)$.
But $\mu\in\overline{O(\boldsymbol\lambda)}$, hence there exists $\boldsymbol\lambda'\in O(\boldsymbol\lambda)$ such that $\boldsymbol\lambda'\in\Phi_{g^*}^{-1}(U)$.
This means that $\Phi_{g^*}(\boldsymbol\lambda')=\boldsymbol\lambda'g^*\in U$.
Hence $\boldsymbol\lambda'g^*\in O(\boldsymbol\lambda)\cap U$ ensuring that $O(\boldsymbol\lambda)\cap U$ is non-empty as required.
\end{proof}

Much more can be said using the theory of algebraic groups (\cite{Borel,Geck2003,Humphrays1991}).
In particular, orbits are locally closed, see~\cite[Proposition~6.7 of Chapter~II]{Borel}.


\begin{definition}\label{DefDegeneration}
Let $\mathfrak{g}_1, \mathfrak{g}_2\in\AnF$.
We say that $\mathfrak{g}_1$ degenerates to $\mathfrak{g}_2$ (respectively, $\mathfrak{g}_1$ properly degenerates to $\mathfrak{g}_2$)
if there exist structure vectors $\boldsymbol\lambda_1$ of $\mathfrak{g}_1$ and $\boldsymbol\lambda_2$ of $\mathfrak{g}_2$ such that $\boldsymbol\lambda_2\in\overline{O(\boldsymbol\lambda_1)}$
(respectively, $\boldsymbol\lambda_2\in\overline{O(\boldsymbol\lambda_1)}\setminus O(\boldsymbol\lambda_1)$).
\end{definition}

In the following remark we include some observations which are immediate from this definition.

\begin{remark}\label{RemarkDegenAlg=OrbitClosStructVectorsDegenIsomClasses}
Let $\mathfrak{g}_1, \mathfrak{g}_2\in\AnF$.

(i) We have that $\mathfrak{g}_1$ degenerates to $\mathfrak{g}_2$ if, and only if, $\boldsymbol\mu_2\in\overline{O(\boldsymbol\mu_1)}$ whenever
$\boldsymbol\mu_1$ is a structure vector of $\mathfrak{g}_1$ and $\boldsymbol\mu_2$ is a structure vector of $\mathfrak{g}_2$
(similarly for proper degenerations).

(ii) Suppose that $\mathfrak{g}_1^*,\mathfrak{g}_2^*\in\AnF$ are such that $\mathfrak{g}_1^*\cong\mathfrak{g}_1$ and $\mathfrak{g}_2^*\cong\mathfrak{g}_2$.
It then follows that $\mathfrak{g}_1^*$ degenerates (resp., properly degenerates) to $\mathfrak{g}_2^*$ whenever $\mathfrak{g}_1$ degenerates (resp., properly degenerates) to $\mathfrak{g}_2$.

\end{remark}

It is then clear from the above definition and remark that the notion of degeneration
(since it only depends on the isomorphism classes of the algebras involved)
can be extended to cover all $\F$-algebras and not just algebra structures on our vector space~$V$.

\begin{example}\label{ExampleFiniteFieldOnePointCasen=2}
(i)
(The case $\F$ is a finite field.)
If $\F$ is finite, it is then obvious that any orbit in $\F^{n^3}$ under the action of $G$ we are considering consists of a finite set of points.
But points are closed sets in the Zariski topology.
Consequently, when $\F$ is finite, all orbits in~$\F^{n^3}$ are closed in the Zariski topology.
We conclude that there exist no proper degenerations over finite fields.

(ii)
(The case $n=2$.)
Let $\boldsymbol\lambda=(\lambda_{ijk})_{1\le i,j,k\le 2}\in\F^{2^3}$ where $\lambda_{121}=1_{\F}=-\lambda_{211}$, $\lambda_{122}=1_{\F}=-\lambda_{212}$
(all other $\lambda_{ijk}$ equal to $0_{\F}$).
(Then $\boldsymbol\lambda$ is a structure vector for $\mathfrak{g}=(V,[,])\in\AnF$ relative to basis $(e_1,e_2)$ of $V$ where
$[e_1,e_2]=e_1+e_2=-[e_2,e_1]$ and $[e_1,e_1]=0_V=[e_2,e_2]$.
Note that $\boldsymbol\lambda$ is a Lie algebra structure vector.)
Consider the change of basis $e_1'=\alpha e_1$, $e_2'=\beta e_2$ where $\alpha$, $\beta$ are both nonzero elements of $\F$.
The products between the elements of this new basis are given by:
$[e_1',e_1']=0_V=[e_2',e_2']$,
$[e_1',e_2']=\alpha\beta[e_1,e_2]=\alpha\beta(e_1+e_2)=\alpha\beta(\alpha^{-1}e_1'+\beta^{-1}e_2')=\beta e_1'+\alpha e_2'$,
$[e_2',e_1']=-\beta e_1'-\alpha e_2'$.
If, instead, we set $e_1'=\beta(e_1+e_2)$, $e_2'=\beta e_2$ with $\beta\ne0_{\F}$ (resp., $e_1'=\alpha e_1$, $e_2'=\alpha(e_1+e_2)$ with $\alpha\ne0_{\F}$), we get that $(0_{\F},0_{\F},\beta,0_{\F},-\beta,0_{\F},0_{\F},0_{\F})\in O(\boldsymbol\lambda)$
(resp., $(0_{\F},0_{\F},0_{\F},\alpha,0_{\F},-\alpha,0_{\F},0_{\F})\in O(\boldsymbol\lambda)$).
Consequently, $\mathcal{K}_2(\F)=\{(0_{\F},0_{\F},\beta,\alpha,-\beta,-\alpha,0_{\F},0_{\F}):\ \alpha,\beta\in\F\}
\subseteq O(\boldsymbol\lambda)\cup \{{\bf 0}\}\subseteq\mathcal{L}_2(\F)\subseteq\mathcal{K}_2(\F)$.
We conclude that $\mathcal{K}_2(\F)= O(\boldsymbol\lambda)\cup \{{\bf 0}\}$ and hence it is a union of two orbits.
But $\mathcal{K}_2(\F)$ is an irreducible subset of $\F^{n^3}$ if $\F$ is infinite
(this follows, for example, from~\cite[Example~1.1.3 and Remark~1.3.2]{Geck2003}) verifying
 that ${\bf 0}\in\overline{O(\boldsymbol\lambda)}$ under the assumption that $\F$ is infinite.
Moreover, the above argument shows that $\mathcal{K}_2(\F)$ is an irreducible $\F G$-submodule of $\F^{2^3}$
for any field $\F$ (including finite fields).
\end{example}

In view of the above example, for the rest of this section and for the whole of Section~\ref{SectionLevel1} (which deal with $\SknF$) we will assume that $\F$ is an arbitrary infinite field and that $n\ge 3$.


\bigskip
Our next aim is to prove Lemma~\ref{LemmaTechnique*} which will play some part in the proof of our main result.

\begin{definition}\label{DefTDegF}
Let $S\subseteq T=\{(i,j,k)\in\mathbb Z^3:\ 1\le i,j,k\le n\}$.
We define $\deg_Sf$ for a non-zero monomial $f\in\F[X_{i,j,k}:\, (i,j,k)\in T]$ to be the degree of the monomial obtained from $f$ by specializing to $1_{\F}$
each of the indeterminates $X_{i,j,k}$ with $(i,j,k)\in T\setminus S$ and which occur in $f$.
In particular, $\deg_Tf=\deg f$.
\end{definition}

\begin{definition}
Let $\hat q=(q_1,\ldots, q_n)\in\mathbb Z^n$ be given.
Also let $T$ be as in Definition~\ref{DefTDegF}.\\
(i) For each $r\in\mathbb Z$, define the subset $S(\hat q,r)$ of $T$ by $S(\hat q,r)=\{(i,j,k)\in T: q_i+q_j-q_k=r\}$.\\
(ii) Define the $\hat q$-auxiliary degree of a nonzero monomial $f\in\F[X_{ijk}:\ (i,j,k)\in T]$ by
\[
\qadeg(f)
=\left\{
\begin{array}{ll}
\sum_{r\ge0}r\deg_{S(\hat q,r)}f, &\mbox{\footnotesize if no indeterminate $X_{ijk}$ with $(i,j,k)\in \cup_{r<0}S(\hat q,r)$ occurs in $f$,}\\
0,&\mbox{\footnotesize otherwise}.
\end{array}
\right.
\]
(iii) Given a structure vector $\boldsymbol{\lambda}=(\lambda_{ijk})\in \F^{n^3}$, define a new structure vector $\boldsymbol\lambda(\hat q)=(\lambda_{ijk}(\hat q))\in \F^{n^3}$ by
\[
\lambda_{ijk}(\hat q)
=\left\{
\begin{array}{ll}
\lambda_{ijk}, &\mbox{\footnotesize if $(i,j,k)\in \cup_{r\le0}S(\hat q,r)$,}\\
0_{\F},&\mbox{\footnotesize otherwise}.
\end{array}
\right.
\]
\end{definition}

It is clear from the above definition that, given $\hat q\in\mathbb Z^n$, only finitely many of the $S(\hat q,r)$ are nonempty as $r$ runs through $\mathbb Z$, and $T$ is the disjoint union of these nonempty $S(\hat q,r)$.

Moreover, nonzero constants (which are monomials of degree zero) have $\hat q$-auxiliary degree zero for any choice of $\hat q$.
However, in general, it is possible for a nonzero monomial $f$ (with $\deg f>0$) to have $\hat q$-auxiliary degree zero for a certain choice of
$\hat q\in\mathbb Z^n$, even though none of the indeterminates $X_{ijk}$ with $(i,j,k)\in \cup_{r<0}S(\hat q,r)$ occurs in $f$.
(These are precisely the monomials which are composed solely by indeterminates $X_{ijk}$ with $(i,j,k)\in S(\hat q, 0)$, provided
$S(\hat q, 0)\neq\varnothing$.)

\begin{example}\label{ExampleQAuxDeg}
(i)
Let $\hat q=(q_i)_{i=1}^n\in\mathbb Z^n$ with $q_i=1$ for $1\le i\le n$.
Then $S(\hat q, r)=\varnothing$ for $r\ne1$ and $T=S(\hat q,1)$.
Then, for a nonzero monomial $f\in\F[X_{ijk}:\ (i,j,k)\in T]$, we have $\qadeg f=\deg f$.

(ii)
Fix $m\in \mathbb Z$ with $1\le m<n$ and let $\hat q=(q_i)_{i=1}^n\in\mathbb Z^n$ with $q_i=0$ for $1\le i\le m$ and $q_i=1$ for $m+1\le i\le n$.
Then $S(\hat q, r)=\varnothing$ for all $r\in\mathbb Z\setminus\{-1,0,1,2\}$.
We also have $S(\hat q,-1)=\{(i,j,k)\in T:$ $1\le i,j\le m$ and $m+1\le k\le n\}$ and
$S(\hat q,0)=\{(i,j,k)\in T:$ $1\le i,j,k\le m\}\cup\{(i,j,k)\in T:$ $1\le i\le m$ and $m+1\le j, k\le n\}\cup
\{(i,j,k)\in T:$ $1\le j\le m$ and $m+1\le i, k\le n\}$.

(iii)
Suppose $n\ge3$ and let $\hat q=(q_i)_{i=1}^n\in\mathbb Z^n$ with $q_i=1$ for $1\le i\le 2$ and $q_i=2$ for $i\ge3$.
Here $S(\hat q,0)=\{(i,j,k)\in T:$ $1\le i,j\le2$ and $3\le k\le n\}$ and
$S(\hat q, r)=\varnothing$ for all $r\in\mathbb Z\setminus\{0,1,2,3\}$.
\end{example}

\begin{remark}\label{RemarqAuxDegreeDecomp}
Given $\hat q\in\mathbb Z^n$ and a nonzero polynomial $h\in\F[X_{ijk}:\ (i,j,k)\in T]$, it will be useful to consider the (unique) decomposition of $h$
(relative to $\hat q$) in the form $h=h_0+h_1+\ldots+h_m$ (with $h_l\in\F[X_{ijk}:\ (i,j,k)\in T]$ for $0\le l\le m$, $h_m\ne0$), where,
for each $l=0,1,\ldots,m$, the polynomial $h_l$ is the sum of all nonzero monomials of $\hat q$-auxiliary degree $l$ occurring in $h$.
(If for some $l$, with $0\le l\le m-1$, no monomial of $\hat q$-auxiliary degree $l$ occurs in $h$, then $h_l$ is taken to be the zero polynomial.)
\end{remark}

\begin{lemma}\label{LemmaTechnique*}
Let $\hat q=(q_i)_{i=1}^n\in\mathbb Z^n$ and let $\boldsymbol\lambda=(\lambda_{ijk})\in\F^{n^3}$ (where $\F$ is infinite).
Suppose further that $\lambda_{ijk}=0_{\F}$ whenever
$(i,j,k)\in\cup_{r<0}S(\hat q,r)$.
Then $\boldsymbol\lambda(\hat q)\in\overline{O(\boldsymbol\lambda)}$.
(In particular, the hypothesis of the lemma is satisfied regardless of the choice of $\boldsymbol\lambda$ by all $\hat q\in\mathbb Z^n$ such
that $\cup_{r<0}S(\hat q,r)=\varnothing$.)
\end{lemma}
\begin{proof}
Let $\hat q=(q_i)_{i=1}^n\in\mathbb Z^n$ and suppose the hypothesis of the lemma is satisfied for $\boldsymbol\lambda=(\lambda_{ijk})\in\F^{n^3}$
(that is, $\lambda_{ijk}=0_{\F}$ whenever $(i,j,k)\in\cup_{r<0}S(\hat q,r)$).
Temporarily fix $\tau\in\F^*$ 
and let $\boldsymbol\lambda(\tau)=(\lambda_{ijk}(\tau))\in\F^{n^3}$ where $\lambda_{ijk}(\tau)=\tau^{q_i+q_j-q_k}\lambda_{ijk}$.
(As usual, we take $\tau^0=1_{\F}$ and for $m\in\mathbb Z$, $m<0$, we take $\tau^m=(\tau^{-1})^{-m}$ with $\tau^{-1}$ being the (unique) multiplicative inverse of $\tau$ in $\F^*$, where $\F^*$ denotes the set of invertible elements of~$\F$.) 

Observe that if $\boldsymbol\lambda$ is the structure vector of the algebra $\mathfrak{g}\in\AnF$ relative to the ordered $\F$-basis $(e_1,\ldots,e_n)$ of $V$,
then $\boldsymbol\lambda(\tau)$ is the structure vector of $\mathfrak{g}$ relative to a new basis $(e_1(\tau),\ldots,e_n(\tau))$ of $V$ defined by
$e_i(\tau)=\tau^{q_i}e_i$ for $1\le i\le n$.
It follows that $\boldsymbol\lambda(\tau)\in O(\boldsymbol\lambda)$.
It is also easy to see that for each $r\in\mathbb Z$, we have $\lambda_{ijk}(\tau)=\tau^r\lambda_{ijk}$ if, and only if, $(i,j,k)\in S(\hat q,r)$.
Consequently, our hypothesis that $\lambda_{ijk}=0_{\F}$ whenever $(i,j,k)\in\cup_{r<0}S(\hat q,r)$ ensures that $\lambda_{ijk}(\tau)=\lambda_{ijk}$
whenever $(i,j,k)\in\cup_{r\le0}S(\hat q,r)$.

Now let $h$ be any nonzero polynomial in the vanishing ideal of $O(\boldsymbol\lambda)$ and consider the decomposition $h=h_1+\ldots+h_m$ with $h_m\ne0$
described in Remark~\ref{RemarqAuxDegreeDecomp}.
On setting $\alpha_l=\ev_{\boldsymbol\lambda}(h_l)\in\F$ for $0\le l\le m$ we get, in view of the discussion above, that $\ev_{\boldsymbol\lambda(\tau)}(h_l)=\tau^l\ev_{\boldsymbol\lambda}(h_l)=\tau^l\alpha_l$
for $0\le l\le m$.
Invoking the fact that $\boldsymbol\lambda(\tau)\in O(\boldsymbol\lambda)$ we get that
$0_{\F}=\ev_{\boldsymbol\lambda(\tau)}(h)=\ev_{\boldsymbol\lambda(\tau)}(h_0)+\ev_{\boldsymbol\lambda(\tau)}(h_1)+\ldots +\ev_{\boldsymbol\lambda(\tau)}(h_m) =\alpha_0+\alpha_1\tau+\ldots+\alpha_m\tau^m$.
This last equality is true independent of the original choice of $\tau\in\F^*$ 
so the fact that $\F$ is an infinite field ensures that
$\alpha_0=\alpha_1=\ldots=\alpha_m=0_{\F}$.
Next observe that it is immediate from the definition of $\boldsymbol\lambda(\hat q)$ and the hypothesis in the statement of the lemma we have assumed, that
$\ev_{\boldsymbol\lambda(\hat q)}(h_l)=0_{\F}$ for $1\le l\le m$ and that $\ev_{\boldsymbol\lambda(\hat q)}(h_0)=\ev_{\boldsymbol\lambda}(h_0)$.
{
[To see this last point there are two cases to consider:
(i) a nonzero monomial $f$ occurring in $h_0$ contains an indeterminate $X_{ijk}$ with $(i,j,k)\in \cup_{r<0}S(\hat q,r)$
and
(ii) a nonzero monomial $f$ occurring in $h_0$ does not contain an indeterminate $X_{ijk}$ with $(i,j,k)\in \cup_{r<0}S(\hat q,r)$ in which case it also does not contain $X_{ijk}$
with $(i,j,k)\in \cup_{r>0}S(\hat q,r)$.
Then clearly $\ev_{\boldsymbol\lambda(\hat q)}f=\ev_{\boldsymbol\lambda}f$ in both cases.]
}

We conclude that $\ev_{\boldsymbol\lambda(\hat q)}(h)=\ev_{\boldsymbol\lambda(\hat q)}(h_0)=\ev_{\boldsymbol\lambda}(h_0)=\alpha_0=0_{\F}$ and this is enough to complete the proof
that $\boldsymbol\lambda(\hat q)\in\overline{O(\boldsymbol\lambda)}$.
\end{proof}

\begin{example}\label{ExampleHomPolynAndDegenToIdeal}
(i)
Let $\hat q=(q_i)_{i=1}^n\in\mathbb Z^n$ where $q_i=1$, $1\le i\le n$.
We have seen in Example~\ref{ExampleQAuxDeg}(i) that $\cup_{r<0}S(\hat q,r)=\varnothing$.
Moreover, $\boldsymbol\lambda(\hat q)={\bf0}$ (the structure vector of the abelian Lie algebra~$\mathfrak{a}_n$) for any $\boldsymbol\lambda\in\F^{n^3}$.
It is now immediate from Lemma~\ref{LemmaTechnique*} that any ($n$-dimensional) algebra $\mathfrak g\in\AnF$ degenerates to $\mathfrak a_n$
(a well-known result).
It also follows from the proof of Lemma~\ref{LemmaTechnique*} (with the choice of $\hat q$ we have made here) that the vanishing ideal of $O(\boldsymbol\lambda)$ is generated by homogeneous polynomials for any $\boldsymbol\lambda\in\F^{n^3}$.

(ii)
Suppose $\mathfrak{g}\in\AnF$ has an $m$-dimensional ideal $\mathfrak{g}_1$ for some $m$ with $1\le m<n$.
{\small\it
}
Pick a basis $(e_1,\ldots,e_n)$ for $V$ by completing a basis $(e_1,\ldots,e_m)$ of the underlying space for $\mathfrak{g}_1$ and let $\boldsymbol\lambda$ be the structure vector of $\mathfrak{g}$ relative to $(e_1,\ldots,e_n)$.
Also let $\hat q=(q_i)_{i=1}^n\in\mathbb Z^n$ with $q_i=0$ for $1\le i\le m$ and $q_i=1$ for $m+1\le i\le n$.
Comparing with Example~\ref{ExampleQAuxDeg}(ii) we see that with this choice of $\boldsymbol\lambda$ and $\hat q$ the hypothesis of Lemma~\ref{LemmaTechnique*} is satisfied.
Also observe that $\lambda_{rst}=0_{\F}$ whenever $(r,s,t)\in \{(i,j,k)\in T:$ $1\le i\le m$ and $m+1\le j, k\le n\}\cup
\{(i,j,k)\in T:$ $1\le j\le m$ and $m+1\le i, k\le n\}$ since $\mathfrak{g}_1$ is an ideal.
We conclude that $\boldsymbol\lambda(\hat q)\in\overline{O(\boldsymbol\lambda)}$.
But $\boldsymbol\lambda(\hat q)$ is a structure vector for $\mathfrak{g_2}\in\AnF$, where $\mathfrak{g_2}$ is isomorphic to $\mathfrak{g}_1\oplus\mathfrak{a}_{n-m}$ (algebra direct sum).
Later on we will see that this result is not true in general if we only assume that $\mathfrak{g}_1$ is a subalgebra of $\mathfrak{g}$.
\end{example}

\subsection{Some necessary conditions}\label{SubsectionNecessaryCond}

Various authors (see for example~\cite{GrunewaldOHalloran1988a,NesterenkoPopovych2006,Rakhimov2005,Seeley1990,Seeley1991}) have considered necessary conditions for degenerations within special classes of algebras.
In this subsection we consider some of these conditions in the general case of algebras over an arbitrary field.

\begin{definition}\label{DefLeftAnnih}
Let $\mathfrak{g}=(V, [,])\in\AnF$.
Define the left annihilator of~$\mathfrak{g}$ by
$
\ann_L\mathfrak{g}=\{c\in V:\ [c,a]=0_{V} \mbox{ for all } a\in V\}.
$
Similarly we can define the right annihilator $\ann_R\mathfrak{g}$.
The two-sided annihilator of $\mathfrak{g}$ is defined by $\ann\mathfrak{g}=\ann_L\mathfrak{g}\cap\ann_R\mathfrak{g}$.
\end{definition}

Note that $\ann_L\mathfrak{g}$, $\ann_R\mathfrak{g}$ and $\ann\mathfrak{g}$ are all $\F$-subspaces of~$V$.
In fact $\ann\mathfrak{g}$ is an ideal of $\mathfrak{g}$.
Also observe that if $\mathfrak{g}\in\SknFa$ then $\ann_L\mathfrak{g}=\ann\mathfrak{g}=\ann_R\mathfrak{g}$.
In particular, $\ann_L\mathfrak{g}=Z(\mathfrak{g})$, the center of~$\mathfrak{g}$, when $\mathfrak{g}\in\LnFa$.

\begin{remark}\label{RemarkDimAnnDimZ<=n-2}
If $\mathfrak{g}=(V,[,]_{\mathfrak{g}})\in\SknFa$ then $\dim_{\F}(\ann\mathfrak{g})\ne n-1$.
(If $\dim_{\F} \ann (\mathfrak{g})=n-1$ we can then complete a basis $(e_i)_{i=1}^{n-1}$ of $\ann\mathfrak{g}$ to a basis $(e_i)_{i=1}^{n}$ of $V$.
But then $e_n\in \ann\mathfrak{g}$ since $[e_n,e_i]_{\mathfrak{g}}=0_V=[e_i,e_n]_{\mathfrak{g}}$ for all $1\le i\le n$, which is a contradiction.)
In particular, whenever $\mathfrak{g}\in\LnFa$ with $\mathfrak{g}\ne\mathfrak{a}_n$ then $\dim_{\F}Z(\mathfrak{g})\le n-2$.
\end{remark}

\begin{result}\label{ResultMinorsRank} (see~\cite[Theorem 19.11, page~153]{Curtis1984})
Let $r,s,t\in\mathbb Z$ with $r\ge1$, $s\ge1$ and $t\ge0$.
Also let $\tilde a\in\F^{r\times s}$.
Then $\rank \tilde a\le t$ if, and only if, all $k$-rowed minor determinants of $\tilde a$ are zero for $k>t$.
\end{result}

\begin{definition}\label{DefStructConstMatrix}
Let $\boldsymbol\lambda=(\lambda_{ijk})\in\F^{n^3}$.
We define the $n\times n^2$ matrices $\tilde a(\boldsymbol\lambda)=(\tilde\alpha_{lm})_{1\le l\le n, 1\le m\le n^2}$ and
$\tilde b(\boldsymbol\lambda)=(\tilde\beta_{lm})_{1\le l\le n, 1\le m\le n^2}$ in $\F^{n\times n^2}$ as follows:
The coefficient $\tilde\alpha_{lm}$ (resp., $\tilde\beta_{lm}$) is the entry in position $(l-1)n^2+m$ (resp., $(m-1)n+l$) of the ordered $n^3$-tuple $(\lambda_{ijk})$ (relative to the ordering we have fixed at the beginning).
\end{definition}

For example, for $n=3$ we have that
{\footnotesize
\[
\tilde a(\boldsymbol\lambda)=\left[
\begin{array}{lllllllll}
\lambda_{111} &\lambda_{112} &\lambda_{113} &\lambda_{121} &\lambda_{122} &\lambda_{123}&\lambda_{131} &\lambda_{132} &\lambda_{133}\\
\lambda_{211} &\lambda_{212} &\lambda_{213} &\lambda_{221} &\lambda_{222} &\lambda_{223}&\lambda_{231} &\lambda_{232} &\lambda_{233}\\
\lambda_{311} &\lambda_{312} &\lambda_{313} &\lambda_{321} &\lambda_{322} &\lambda_{323}&\lambda_{331} &\lambda_{332} &\lambda_{333}
\end{array}
\right].
\]
}
Note that $\boldsymbol\lambda=(\lambda_{ijk})$ can be recovered from $\tilde a(\boldsymbol\lambda)$ by placing its rows one next to the other (starting from the first row and ending with the last one).
Also, for $n=3$,
{\footnotesize
\[
\tilde b(\boldsymbol\lambda)=\left[
\begin{array}{lllllllll}
\lambda_{111} &\lambda_{121} &\lambda_{131} &\lambda_{211} &\lambda_{221} &\lambda_{231}&\lambda_{311} &\lambda_{321} &\lambda_{331}\\
\lambda_{112} &\lambda_{122} &\lambda_{132} &\lambda_{212} &\lambda_{222} &\lambda_{232}&\lambda_{312} &\lambda_{322} &\lambda_{332}\\
\lambda_{113} &\lambda_{123} &\lambda_{133} &\lambda_{213} &\lambda_{223} &\lambda_{233}&\lambda_{313} &\lambda_{323} &\lambda_{333}
\end{array}
\right].
\]
}
Note that the ``transpose'' of  $\boldsymbol\lambda$ can be recovered from $\tilde b(\boldsymbol\lambda)$ by placing the columns of $\tilde b(\boldsymbol\lambda)$ one below the other (starting from the first column and ending with the last one).

\begin{remark}\label{RemarkRankTildeA<=tIsClosed}
It is then immediate from Result~\ref{ResultMinorsRank} that the sets $\{\boldsymbol\lambda\in\F^{n^3}:\ \rank\tilde a(\boldsymbol\lambda)\le t\}$ and $\{\boldsymbol\lambda\in\F^{n^3}:\ \rank\tilde b(\boldsymbol\lambda)\le t\}$ are both closed subsets of~$\F^{n^3}$ in the Zariski topology for any nonnegative integer~$t$.
\end{remark}

For $\mathfrak{g}=(V,[,]_{\mathfrak{g}})\in\AnF$ we denote by $\mathfrak{g}^2$ the $\F$-subspace of $\mathfrak{g}$ spanned by all products of the form $[x,y]_{\mathfrak{g}}$ with $x,y\in V$. 

\begin{lemma}\label{LemmaDimAnnDimg2}
Let $\mathfrak{g}=(V,[,])\in\AnF$ and let $\boldsymbol\lambda$ be any structure vector of $\mathfrak{g}$.
Then, \\
(i)
$\dim_{\F}(\ann_L\mathfrak{g})=n-\rank\tilde a(\boldsymbol\lambda)$
\\
(ii)
$\dim_{\F}\mathfrak{g}^2=\rank\tilde b(\boldsymbol\lambda)$
\end{lemma}
\begin{proof}
Assume the hypothesis.
Then there exists an $\F$-basis $(e_i)_{i=1}^n$ of $V$ relative to which $\boldsymbol\lambda$ is the structure vector of $\mathfrak{g}$.

(i)
Let $\Psi:\F^n\to\F^{n^2}:$ $\tilde x\mapsto\tilde x\tilde a(\boldsymbol\lambda)$, ($\tilde x\in\F^n$).
(Here we identify $\F^r$ with $\F^{1\times r}$.)
It is then easy to see that $\alpha_1e_1+\ldots+\alpha_ne_n\in\ann_L\mathfrak{g}$ ($\alpha_i\in\F$) if, and only if, $(\alpha_1,\ldots,\alpha_n)\in\ker\Psi$.
Moreover, $\dim_{\F}\im\Psi=\rank\tilde a(\boldsymbol\lambda)$ (since $\im\Psi$ is the $\F$-span of the rows of $\tilde a(\boldsymbol\lambda)$).
We conclude that $\dim_{\F}(\ann_L\mathfrak{g})=\dim_{\F}\ker\Psi=n-\rank\tilde a(\boldsymbol\lambda)$.

(ii)
This is immediate once we observe that $\mathfrak{g}^2$ is the $\F$-span of all the products of the form $[e_i,e_j]$ for $1\le i,j\le n$.
\end{proof}

\begin{lemma}\label{LemmaDimAnnIncreasesDimSquareDecreases}
Let $\mathfrak{g}, \mathfrak{g_1}\in\AnF$ and suppose that $\mathfrak{g}$ degenerates to $\mathfrak{g_1}$.
Then,\\
 $\dim_{\F}(\ann_L\mathfrak{g}_1)\ge\dim_{\F}(\ann_L\mathfrak{g})$ and $\dim_{\F}\mathfrak{g}_1^2\le\dim_{\F}\mathfrak{g}^2$.
\end{lemma}

\begin{proof}
Assume the hypothesis and let $\boldsymbol\lambda$ and $\boldsymbol\nu$ be structure vectors of $\mathfrak{g}$ and $\mathfrak{g_1}$ respectively.
Also let $S_1(\boldsymbol\lambda)=\{\boldsymbol\mu\in\F^{n^3}:\ \rank\tilde a(\boldsymbol\mu)\le\rank\tilde a(\boldsymbol\lambda)\}$ and
$S_2(\boldsymbol\lambda)=\{\boldsymbol\mu\in\F^{n^3}:\ \rank\tilde a(\boldsymbol\mu)\le\rank\tilde b(\boldsymbol\lambda)\}$.
In the argument below $S(\boldsymbol\lambda)$ can be any one of $S_1(\boldsymbol\lambda)$, $S_2(\boldsymbol\lambda)$.
It is clear from Lemma~\ref{LemmaDimAnnDimg2} that $S(\boldsymbol\lambda)$ is a union of orbits, in particular $O(\boldsymbol\lambda)\subseteq S(\boldsymbol\lambda)$.
Moreover, the set $S(\boldsymbol\lambda)$ is Zariski-closed (see Remark~\ref{RemarkRankTildeA<=tIsClosed}).
It follows that $\overline{O(\boldsymbol\lambda)}\subseteq S(\boldsymbol\lambda)$ and hence $\boldsymbol\nu\in S(\boldsymbol\lambda)$ since we have assumed that $\boldsymbol\nu\in \overline{O(\boldsymbol\lambda)}$.
Hence $\rank\tilde a(\boldsymbol\nu)\le\rank\tilde a(\boldsymbol\lambda)$ and $\rank\tilde b(\boldsymbol\nu)\le\rank\tilde b(\boldsymbol\lambda)$.
Invoking Lemma~\ref{LemmaDimAnnDimg2} again we get the desired result.
\end{proof}

\section{Orbit closures in $\SknF$ consisting of precisely two orbits}\label{SectionLevel1}

We continue with our assumption that $n\ge3$ and $\F$ is an arbitrary infinite field.

\medskip

In this section we aim at providing a proof of our main Theorem~\ref{TheoremMain}.
This is achieved via a sequence of lemmas.
We also discuss how this theorem can be used in order to obtain information about the composition series of the $\F G$-module $\SknF$.
\medskip

First we introduce the algebra structures $\mathfrak{r}_n$ and $\mathfrak{h}_n\in\LnFa$.

\vspace{-1.5ex}
\begin{definition}\label{DefHeisenbergRemA}
Let $\boldsymbol\rho=(\rho_{ijk})$ and $\boldsymbol\eta=(\eta_{ijk})\in\F^{n^3}$ be such that the only nonzero components of $\boldsymbol\rho$ are
$\rho_{ini}=1_{\F}=-\rho_{nii}$ for $1\le i\le n-1$ and the only nonzero components of $\boldsymbol\eta$ are $\eta_{123}=1_{\F}=-\eta_{213}$.
It is easy to observe that $\boldsymbol\rho,\boldsymbol\eta\in\LnF$.
Now define $\mathfrak{r}_n,\mathfrak{h}_n\in\LnFa$ by $\mathfrak{r}_n=\Theta^{{}^{-1}}(\boldsymbol\rho)$ and $\mathfrak{h}_n=\Theta^{{}^{-1}}(\boldsymbol\eta)$ with $\Theta$ as in Remark~\ref{RemarkAnFFn3AreVectorSpaces}.
\end{definition}

We have chosen the simpler notation $\boldsymbol\rho$, $\boldsymbol\eta$ instead of the more precise $\boldsymbol\rho_n$, $\boldsymbol\eta_n$ as $n$ will not vary within our arguments (we work within a fixed $n$ with $n\ge3$).

\begin{remark}\label{RemarkHeisenbergRemA}
Keeping the notation of the previous definition, we see that $\boldsymbol\rho$ is the structure vector of the algebra $\mathfrak{r}_n=(V,[,]_{\mathfrak{r}_n})\in\LnFa$ relative to the basis $(v_i^*)_{i=1}^n$ of $V$ we have fixed, where the only nonzero products between elements of this basis are $[v_i^*,v_n^*]_{\mathfrak{r}_n}=v_i^*=-[v_n^*,v_i^*]_{\mathfrak{r}_n}$ ($1\le i\le n-1$).

Moreover, $\boldsymbol\eta$ is the structure vector, again relative to the basis $(v_i^*)_{i=1}^n$ of $V$, for the algebra $\mathfrak{h}_n=(V,[,]_{\mathfrak{h}_n})\in\LnFa$, where the nonzero products between the elements of this basis are $[v_1^*,v_2^*]_{\mathfrak{h}_n}=v_3^*=-[v_2^*,v_1^*]_{\mathfrak{h}_n}$.
Note that $\mathfrak h_3$ is isomorphic to the Heisenberg algebra and that $\mathfrak{h}_n$ is isomorphic to the Lie algebra direct sum $\mathfrak h_3\oplus\mathfrak{a}_{n-3}$.
\end{remark}

We denote by $\Fspan(x_1,\ldots,x_n)$ the set of $\F$-linear combinations of the elements $x_1,\ldots,x_n$ of $V$.

\begin{definition}\label{DefCond*}
Let $\mathfrak{g}=(V,[,]_{\mathfrak{g}})\in\AnF$.\\
(i) We say that $\mathfrak{g}$ satisfies condition $(*)$ if $[x,y]_{\mathfrak{g}}\in\Fspan(x,y)$ for all $x,y\in V$.\\
(ii) We say that $\mathfrak{g}$ satisfies condition $(**)$ if $[x,x]_{\mathfrak{g}}\in\Fspan(x)$ for all $x\in V$.
\end{definition}

It is then immediate that every $\mathfrak{g}\in\SknFa$ satisfies condition $(**)$.
Moreover, $\mathfrak{g}\in\AnF$ satisfies condition $(**)$ whenever $\mathfrak{g}$ satisfies condition $(*)$ but the converse is not true in general.

It is also clear from the above definition that if $\mathfrak{g},\mathfrak{g}_1\in\AnF$ are isomorphic and in addition $\mathfrak{g}$ satisfies condition $(*)$, then $\mathfrak{g}_1$ also satisfies condition $(*)$.
Hence, the subset $\{\boldsymbol\lambda\in\F^{n^3}:\ \boldsymbol\lambda$ is a structure vector for some $\mathfrak{g}\in\AnF$ that satisfies condition $(*)\}$ of $\F^{n^3}$ is a union of orbits.

Finally we remark that $\mathfrak{h}_n$ does not satisfy condition $(*)$ whereas $\mathfrak{a}_n$, the $n$-dimensional abelian Lie algebra, satisfies condition $(*)$.
Later on we will show that $\mathfrak{r}_n$ also satisfies condition $(*)$ and even more, that any algebra $\mathfrak{g}\in\SknFa$ which satisfies condition $(*)$ is in fact isomorphic to either $\mathfrak{a}_n$ or $\mathfrak{r}_n$.

\bigskip
We first focus attention on algebras not satisfying condition $(*)$.
In the proof of the following lemma we use ideas from the proofs of~\cite[Theorem~5.2]{Lauret2003} and~\cite[Proposition~2.2]{Khudoyberdiyev&Omirov2013}.

\begin{lemma}\label{LemmaNot(*)DegeneratesToHeis}
Let $\mathfrak{g}=(V,[,]_{\mathfrak{g}})\in\AnF$ and suppose that $\mathfrak{g}$ satisfies condition~$(**)$.
Suppose further that $\mathfrak{g}$ does not satisfy condition~$(*)$.
Then $\mathfrak{g}$ degenerates to $\mathfrak h_n$.
(In particular, the hypothesis of the lemma is satisfied by any $\mathfrak{g}\in\SknFa$ which does not satisfy condition~$(*)$.)
\end{lemma}
\begin{proof}
Assume the hypothesis.
Then there exist elements $x,y\in V$ such that $[x,y]\not\in\Fspan( x,y)$.
Observe that the set $\{x,y,[x,y]_{\mathfrak{g}}\}$ is $\F$-linearly independent.
To see this we also need to invoke the fact that $\mathfrak{g}$ satisfies condition~$(**)$ so, in particular, $y$ cannot be a scalar multiple of $x$
(otherwise we would have that $[x,y]_{\mathfrak{g}}\in\Fspan(x)\subseteq\Fspan(x,y)$).
We can thus complete this set to a basis $e_1=x$, $e_2=y$, $e_3=[x,y]_{\mathfrak{g}}$, $e_4,\ldots,e_n$ of $V$
(recall we assume $n\ge3$).
Let $\boldsymbol\lambda=(\lambda_{ijk})\in\F^{n^3}$ be the structure vector of $\mathfrak{g}$ relative to $(e_i)_{i=1}^n$.
Then $\lambda_{123}=1_{\F}$ and $\lambda_{12k}=0_{\F}$ for all $k>3$.
Moreover, $\lambda_{213}=-1_{\F}$ and $\lambda_{21k}=0_{\F}$ for all $k>3$ (this follows from the fact that $[e_1,e_2]_{\mathfrak{g}}+[e_2,e_1]_{\mathfrak{g}}\in\Fspan(e_1,e_2)$ in view of the hypothesis that $\mathfrak{g}$ satisfies condition~$(**)$).
Finally note that $\lambda_{11k}=\lambda_{22k}=0_{\F}$ for all $k\ge3$.

We want to use Lemma~\ref{LemmaTechnique*} in order to complete the proof,
so let $\hat q=(q_i)_{i=1}^n\in\mathbb Z^n$ where $q_i=1$ for $1\le i\le 2$ and $q_i=2$ for $i\ge 3$.
Comparing with Example~\ref{ExampleQAuxDeg}(iii), we see that $\boldsymbol\lambda(\hat q)=\boldsymbol\eta$ and, moreover, the hypothesis of Lemma~\ref{LemmaTechnique*} is satisfied.
We conclude that $\boldsymbol\eta\in\overline{O(\boldsymbol\lambda)}$.
\end{proof}

Next we introduce some more subsets of $\F^{n^3}$.
Let $P=\{\boldsymbol\lambda=(\lambda_{ijk})_{1\le i,j,k\le n}\in \SknF:\ \lambda_{ijk}=0_{\F}$ whenever $k\not\in\{i,j\}$ and $\lambda_{iji}=\lambda_{kjk}$ whenever $j\not\in\{i,k\}\}$.
We also denote by $\boldsymbol\rho(\F G)$ and $\boldsymbol\eta(\F G)$ the $\F G$-submodules of $\F^{n^3}$ generated by $\boldsymbol\rho$ and $\boldsymbol\eta$ respectively.

It is easy to see that $P$ is an $\F$-subspace of $\SknF$ (and also an algebraic subset of $\F^{n^3}$).
To compute $\dim_{\F}P$, first observe that at most $2n(n-1)$ of the `positions' $(i,j,k)\in T$ (with $T$ as in Definition~\ref{DefTDegF}) can afford coefficients $\lambda_{ijk}$ which can possibly be nonzero.
This is because for each fixed $j$ only the coefficients $\lambda_{iji}$ and $\lambda_{ijj}$ (with $i\ne j$) can possibly be nonzero.
Now the condition $\lambda_{iji}=\lambda_{kjk}$ whenever $j\not\in\{i,k\}$ forces $\lambda_{i1i}=\alpha_1$ ($i\ne1$), $\lambda_{i2i}=\alpha_2$ ($i\ne2$), \ldots,
$\lambda_{ini}=\alpha_n$ ($i\ne n$) for some $\alpha_1,\ldots,\alpha_n\in\F$.
Invoking the fact that $\lambda_{ijj}=-\lambda_{jij}=-\alpha_i$ for $j\ne i$, we conclude that $\dim_{\F} P=n$.

\begin{lemma}\label{LemmaPIsOrbitClosureOfMu'}
$P=O(\boldsymbol\rho)\cup\{{\bf0}\}=\overline{O(\boldsymbol\rho)}=\boldsymbol\rho(\F G)$.
\end{lemma}
\begin{proof}
We keep the notation about $\boldsymbol\rho$ and $\mathfrak{r}_n$ we have fixed in Remark~\ref{RemarkHeisenbergRemA}.
In particular, we have that $[v_i^*,v_n^*]_{\mathfrak{r}_n}=v_i^*=-[v_n^*,v_i^*]_{\mathfrak{r}_n}$
($1\le i\le n-1$) are the only nonzero products between the elements of our basis $(v_i^*)_{i=1}^n$.
Consider now the new basis $(u_i)_{i=1}^n$ of $V$ where $u_i=\sum_{j=1}^n\beta_{ji}v_j^*$
(with $\beta_{ij}$ for $1\le i,j\le n$ being the $(i,j)$-component of an invertible $n\times n$ matrix over $\F$).
We then have, for $i\ne j$, that $[u_i,u_j]_{\mathfrak{r}_n}=\beta_{nj}u_i-\beta_{ni}u_j$ (and $[u_i,u_i]_{\mathfrak{r}_n}={\bf0}_{\F}$).
On setting $\alpha_i=\beta_{ni}$ for $i=1,\ldots,n$
(note that at least one of the $\alpha_i$ is nonzero since the matrix $(\beta_{ij})$ is invertible)
and comparing with the discussion immediately before the lemma we get that $P=O(\boldsymbol\rho)\cup\{{\bf0}\}$.
Invoking now the fact that $P$ is an $\F$-subspace of $\SknF$, we get that $P$ is in fact an $\F G$-submodule of $\F^{n^3}$, from which the equality $P=\boldsymbol\rho(\F G)$ follows easily.
Finally, combining the fact that $P$ is Zariski-closed with the fact that ${\bf0}\in\overline{O(\boldsymbol\rho)}$ we get that $P=\overline{O(\boldsymbol\rho)}$.
\end{proof}

We remark in passing that if $\mathfrak{s}$ is the subalgebra of $\mathfrak{r}_n$ having $\Fspan(v_1^*,v_n^*)$ as its underlying vector space, then a consequence of Lemma~\ref{LemmaPIsOrbitClosureOfMu'} is that there is no degeneration from $\mathfrak{r}_n$ to $\mathfrak{s}\oplus\mathfrak{a}_{n-2}$ (Lie algebra direct sum).
However, this is not a counterexample to the result in Example~\ref{ExampleHomPolynAndDegenToIdeal}(ii) as $\mathfrak{s}$ in fact is not an ideal of $\mathfrak{r}_n$.

\begin{corollary}\label{CorolPsatisfiesCond*}
(i)
$\mathfrak{r}_n$ satisfies condition~$(*)$.

(ii)
Let $\boldsymbol\lambda\in\F^{n^3}$.
Then $\boldsymbol\lambda$ belongs to $P$ if, and only if, $\boldsymbol\lambda$ is a structure vector for some algebra $\mathfrak{g}\in\SknFa$ which satisfies
condition~$(*)$.
\end{corollary}
\begin{proof}
(i) Let $x,y\in V$.
Keeping the notation we have fixed in Remark~\ref{RemarkHeisenbergRemA} we see that if $y\in\Fspan(x)$ then $[x,y]_{\mathfrak{r}_n}=0_V\,(\in\Fspan(x,y))$ since $\mathfrak{r}_n\in\SknFa$.
If $y\not\in\Fspan(x)$ then we can complete $\{x,y\}$ to a basis of $V$.
But we have seen in the proof of the previous lemma that for any choice of basis $(u_i')_{i=1}^n$ of $V$ we have $[u_i',u_j']_{\mathfrak{r}_n}\in\Fspan(u_i',u_j')$ for all $i,j$.
It follows that $[x,y]_{\mathfrak{r}_n}\in\Fspan(x,y)$.

(ii) For the `if' part.
Suppose $\boldsymbol\lambda=(\lambda_{ijk})\in\F^{n^3}$ is the structure vector, relative to the basis $(e_i)_{i=1}^n$ of $V$,  of $\mathfrak{g}=(V,[,])\in\SknFa$, where $\mathfrak{g}$ satisfies condition~$(*)$.
It follows from the properties we have assumed for $\mathfrak{g}$ that $\lambda_{iik}=0_{\F}$ and $\lambda_{ijk}=-\lambda_{jik}$ for all $i,j,k$ and that $\lambda_{ijk}=0_{\F}$ whenever $k\not\in\{i,j\}$.
Now for $j\not\in\{k,i\}$ (with $k\ne i$) we have $[e_i,e_j]=\lambda_{iji}e_i+\lambda_{ijj}e_j$ and $[e_k,e_j]=\lambda_{kjk}e_k+\lambda_{kjj}e_j$.
This gives $[(e_i+e_k),e_j]=[e_i,e_j]+[e_k,e_j]=\lambda_{iji}e_i+\lambda_{kjk}e_k+(\lambda_{ijj}+\lambda_{kjj})e_j$.
But from hypothesis $[(e_i+e_k),e_j]\in\Fspan(e_i+e_k,e_j)$.
This forces $\lambda_{iji}=\lambda_{kjk}$ for all $i,j,k$ such that $j\not\in\{k,i\}$ since $e_i$, $e_j$, $e_k$ are three distinct elements of the above basis.
We conclude that $\boldsymbol\lambda\in P$.

The `only if' part is immediate from item~(i) of this corollary and Lemma~\ref{LemmaPIsOrbitClosureOfMu'} since both $\mathfrak{r}_n$ and $\mathfrak{a}_n$ belong to $\LnFa\subseteq\SknFa$.
\end{proof}

Since the algebras $\mathfrak{r}_n$, $\mathfrak{h}_n$ are not isomorphic, we get that $P\cap O(\boldsymbol\eta)=\varnothing$ in view of Lemma~\ref{LemmaPIsOrbitClosureOfMu'}.
Moreover, combining Lemmas~\ref{LemmaNot(*)DegeneratesToHeis}, \ref{LemmaPIsOrbitClosureOfMu'} and Corollary~\ref{CorolPsatisfiesCond*} with the fact that ${\bf0}\in\overline{O(\boldsymbol\lambda)}$ for any $\boldsymbol\lambda\in\F^{n^3}$ we get

\vspace{-2ex}
\begin{corollary}\label{NotPDegeneratesToMu''OtherwiseLev>1}
Let $\boldsymbol\lambda\in\SknF\setminus P$.
Then,

(i)
$\boldsymbol\eta\in\overline{O(\boldsymbol\lambda)}$.

(ii)
If, in addition, $\boldsymbol\lambda\not\in O(\boldsymbol\eta)$, then $\overline{O(\boldsymbol\lambda)}$ contains at least 3 distinct orbits.
\end{corollary}

In the following remark we collect some facts about the algebra $\mathfrak{h}_n$ we will use later.

\vspace{-2ex}
\begin{remark}\label{RemarkHeisenbergPropertiesDimZ<=n-2}
(i)
$\dim_{\F} Z(\mathfrak{h}_n)=n-2$.
(Explanation:
Comparing with the discussion in Remark~\ref{RemarkHeisenbergRemA} we easily see that  $\{v_3^*,\ldots,v_n^*\}\subseteq Z(\mathfrak{h}_n)$,
so $\Fspan(v_3^*,\ldots,v_n^*)\subseteq Z(\mathfrak{h}_n)$.
In order to conclude that $\dim_{\F} Z(\mathfrak{h}_n)=n-2$, it is enough to show that $Z(\mathfrak{h}_n)\subseteq\Fspan(v_3^*,\ldots,v_n^*)$.
But if $z=\alpha_1v_1^*+\ldots+\alpha_nv_n^*\in Z(\mathfrak{h}_n)$ ($\alpha_i\in\F$), then $\alpha_1v_3^*=[z,v_2^*]_{\mathfrak{h}_n}=0_V=[v_1^*,z]_{\mathfrak{h}_n}=\alpha_2v_3^*$ which forces $\alpha_1=0_{\F}=\alpha_2$.)

(ii)
$\mathfrak{h}_n\in \BnFa$.
To see this, note that for $x_1,x_2\in V$, $[x_1,x_2]_{\mathfrak{h}_n}=\alpha v_3^*$ for some $\alpha\in\F$ which in turn shows that $[x_1,x_2]_{\mathfrak{h}_n}\in Z(\mathfrak{h}_n)$.
(Alternatively, we can observe that $\sum_l\eta_{ijl}\eta_{lkm}=0_{\F}$ for $1\le i,j,k,m\le n$ which in turn shows that $\boldsymbol\eta\in\BnF$.)
\end{remark}

\begin{lemma}\label{LemmadimZ=n-2Heisenberg}
Let $\mathfrak{g}=(V,[,])\in\SknFa\cap\BnFa$ with $\dim_{\F} (\ann\mathfrak{g})=n-2$.
Then $\mathfrak{g}\cong\mathfrak{h}_n$.
\end{lemma}
\begin{proof}
Assume the hypothesis.
By extending a basis of $\ann\mathfrak{g}$ to a basis of $V$ we get an $\F$-subspace decomposition
$V=\Fspan(b_1,b_2)\oplus \ann\mathfrak{g}$.
(In particular, the elements $b_1,b_2$ of $V$ do not belong to $\ann\mathfrak{g}$, which also gives that $[b_1,b_2]\ne0_V$.)
From the hypothesis $[b_1,b_2]\in \ann\mathfrak{g}$ since $[[b_1,b_2],x]=0_V=[x,[b_1,b_2]]$ for all $x\in V$. 
By completing $\{[b_1,b_2]\}$ to a basis of $\ann\mathfrak{g}$ we obtain a basis of $V$ of the form $(e_i)_{i=1}^{n}$,
where $e_1=b_1$, $e_2=b_2$, $e_3=[b_1,b_2]$ and $\ann\mathfrak{g}=\Fspan(e_3,\ldots,e_n)$.
(Clearly $[b_1,b_2]\not\in\Fspan(b_1,b_2)$ since $[b_1,b_2]\in\ann\mathfrak{g}\setminus \{0_V\}$.)
Finally, invoking the fact that $g\in\SknFa$ we get that the structure vector of $\mathfrak{g}$ relative to the basis $(e_i)_{i=1}^{n}$ is precisely $\boldsymbol\eta$ which is enough to complete the proof.
\end{proof}

\begin{lemma}\label{LemmaHeisenbergIsLevel1}
The only proper degeneration of $\mathfrak{h}_n$ is to the abelian Lie algebra $\mathfrak{a}_n$.
In other words, $\overline{O(\boldsymbol\eta)}=O(\boldsymbol\eta)\cup\{{\bf0}\}$.
\end{lemma}
\begin{proof}
Let $\mathfrak{g}\in\AnF$, with $\mathfrak{g}\ne\mathfrak{a}_n$, be a proper degeneration of $\mathfrak{h}_n$.
Clearly, $\mathfrak{g}\in\LnFa$ since $\mathfrak{h}_n\in\LnFa$ and as we have seen, $\LnF$ is an algebraic subset of $\F^{n^3}$.
Our assumption $\mathfrak{g}\ne\mathfrak{a}_n$ ensures that $\dim_{\F} Z(\mathfrak{g})\le n-2$ (see Remark~\ref{RemarkDimAnnDimZ<=n-2}).
On the other hand, $\dim_{\F} Z(\mathfrak{g})\ge n-2$ in view of Lemma~\ref{LemmaDimAnnIncreasesDimSquareDecreases} and Remark~\ref{RemarkHeisenbergPropertiesDimZ<=n-2}(i) since there is a degeneration from $\mathfrak{h}_n$ to $\mathfrak{g}$.
We conclude that $\dim_{\F} Z(\mathfrak{g})=n-2$.
Moreover, we have $\mathfrak{g}\in\BnFa$ since $\mathfrak{h}_n\in\BnFa$ and $\BnF$ is also an algebraic subset of $\F^{n^3}$.
Invoking Lemma~\ref{LemmadimZ=n-2Heisenberg} we then get that $\mathfrak{g}\cong\mathfrak{h}_n$, a contradiction.
The desired result now follows from the fact that any algebra structure in~$\AnF$ degenerates to $\mathfrak{a}_n$ (see Example~\ref{ExampleHomPolynAndDegenToIdeal}(i)).
\end{proof}

We can now provide a proof of the first of our two main results.

{\bf Proof of Theorem~\ref{TheoremMain}.}
Combine Lemma~\ref{LemmaPIsOrbitClosureOfMu'}, Corollary~\ref{NotPDegeneratesToMu''OtherwiseLev>1} and Lemma~\ref{LemmaHeisenbergIsLevel1}.
{\hfill $\Box$\par}

\medskip
Finally, we remark that the fact that orbits are locally closed could be used to provide certain simplifications to the above proof.
However, we have chosen a presentation which is more elementary and self-contained.

\subsection{On the composition series of the $\F G$-module $\SknF$}\label{SubsectionCompositionSeries}

\begin{definition}\label{DefAdjMapUnimod}
Let $\mathfrak{g}=(V,[,]_{\mathfrak{g}})\in\AnF$.
Fix $x\in V$.
We define the adjoint map in $\mathfrak{g}$ (relative to $x$) by $\ad_x:V\to V:$ $y\mapsto[x,y]_{\mathfrak{g}}$, ($y\in V$).
Then $\ad_x$ is an $\F$-linear map.
We say that the algebra structure $\mathfrak{g}$ is unimodular if ${\rm trace}(\ad_x)=0_{\F}$ for each $x\in V$.
\end{definition}

\begin{remark}\label{RemarkUnimodularCriterium}
Suppose that $(u_1,\ldots, u_n)$ is an $\F$-basis of~$V$ and that $\mathfrak{g}\in\AnF$.
Observe that if $\boldsymbol\lambda=(\lambda_{ijk})\in\F^{n^3}$ is the structure vector of $\mathfrak{g}$ relative to $(u_1,\ldots,u_n)$,
then ${\rm trace}(\ad_{u_i})=\sum_{j=1}^n\lambda_{ijj}$.
Let $x=\alpha_1u_1+\ldots+\alpha_nu_n\in V$ ($\alpha_i\in\F$).
We then have that ${\rm trace}(\ad_x)=\sum_{i=1}^n\alpha_i{\rm trace}(\ad_{u_i})$.
It follows that $\mathfrak{g}$ is unimodular if, and only if, ${\rm trace}(\ad_{u_i})=0_{\F}$ for $1\le i\le n$, and this last condition holds true if, and only if, $\sum_{j=1}^n\lambda_{ijj}=0_{\F}$ for $1\le i\le n$.
\end{remark}

\begin{definition}\label{DefUnF}
Define $\UnF=\{\boldsymbol\lambda=(\lambda_{ijk})\in\SknF:$ $\sum_{j=1}^n\lambda_{ijj}=0_{\F}$ for $i=1,\ldots,n\}$.
\end{definition}

It is then immediate that $\UnF$ is an $\F$-subspace of $\F^{n^3}$.
Moreover, $\UnF$ is also a union of orbits in view of Remark~\ref{RemarkUnimodularCriterium}.
We can thus regard $\UnF$ as an $\F G$-submodule of $\F^{n^3}$. 

\medskip

Next we compute $\dim_{\F}\UnF$. 
Assume first that $\lambda=(\lambda_{ijk})\in\SknF$.
This forces $\lambda_{iik}=0$ and $\lambda_{ijk}=-\lambda_{jik}$ for $1\le i,j,k\le n$.
We see that this results in $\boldsymbol\lambda$ having $\frac12n^2(n-1)\,(=\dim_{\F}\SknF)$ `free' positions, the components in the remaining positions of $\boldsymbol\lambda$ being forced.
(With $T$ as in Definition~\ref{DefTDegF} we can take for example $T_1=\{(i,j,k)\in T:\ i<j\}$ as our set of `free' positions for $\boldsymbol\lambda$.)
The fact that $\lambda_{ijk}=-\lambda_{jik}$ ensures that $T_2\cup T_3$, where $T_2=\{(i,j,k)\in T:\ i<j$ and $k\not\in\{i,j\}\}$ and $T_3=\{(i,j,k)\in T:\ k=j$ and $k\ne i\}$, can also be taken as the set of `free' positions for $\boldsymbol\lambda\in\SknF$.
We now make the further  assumption that $\boldsymbol\lambda\in\UnF$.
Note that $T_2$ has exactly $\frac12n(n-1)(n-2)$ members and no new constraints are imposed on the corresponding coefficients by our further assumption (giving $\frac12n(n-1)(n-2)$ `free' positions for $\boldsymbol\lambda\in\UnF$ from set $T_2$).
The set $T_3$ has exactly $n(n-1)$ members which can be conveniently partitioned into $n$ disjoint classes each having $(n-1)$ members.
These are the classes $\{(1,j,j):\ j\ne1\}$, $\{(2,j,j):\ j\ne2\}$, \ldots, $\{(n,j,j):\ j\ne n\}$.
Now for $\boldsymbol\lambda$ to be an element of $\UnF$ what is required is precisely the condition that the sum of the coefficients of $\boldsymbol\lambda$ corresponding to the positions within each one of the above classes equals $0_{\F}$.
This gives $n(n-2)$ `free' positions for $\boldsymbol\lambda$ from set $T_3$.
We conclude that $\dim_{\F}\UnF=\frac12n(n-1)(n-2)+n(n-2)=\frac12n(n-2)(n+1)=\dim_{\F}\SknF-n=\dim_{\F}\SknF-\dim_{\F}\boldsymbol\rho(\F G)$.

\medskip

We are now ready to construct a spanning set for $\UnF$ which we will need in the next lemma:
For $1\le r,s,t\le n$ with $r<s$ and $t\not\in\{r,s\}$, define $\hat{\boldsymbol\lambda}(r,s,t)$ to be the the structure vector $(\hat \lambda_{ijk})\in\UnF$ where
\[
\hat \lambda_{ijk}=\left\{
\begin{array}{ll}
1_{\F},& \mbox{if } (i,j,k)=(r,s,t),\\
-1_{\F},& \mbox{if } (i,j,k)=(s,r,t),\\
0_{\F},& \mbox{otherwise}.
\end{array}
\right.
\]
Moreover, for $1\le r,s,t\le n$ with $s\ne t$ and $r\not\in\{s,t\}$ define $\tilde{\boldsymbol\lambda}(r,s,t)$ to be the structure vector $(\tilde\lambda_{ijk})$ where
\[
\tilde\lambda_{ijk}=\left\{
\begin{array}{ll}
1_{\F},& \mbox{if } (i,j,k)=(r,s,s) \mbox{ or } (i,j,k)=(t,r,t),\\
-1_{\F},& \mbox{if } (i,j,k)=(r,t,t) \mbox{ or } (i,j,k)=(s,r,s),\\
0_{\F},& \mbox{otherwise}.
\end{array}
\right.
\]
It is then an easy consequence of the preceding discussion that the union $\{\hat{\boldsymbol\lambda}(r,s,t):\ 1\le r,s,t\le n$ with
$r<s$ and $t\not\in\{r,s\}\}\cup\{\tilde{\boldsymbol\lambda}(r,s,t):\ 1\le r,s,t\le n$ with $s\ne t$ and $r\not\in\{s,t\}\}$ is in fact a spanning set for $\UnF$.
\begin{lemma}\label{LemmaMnF}
$\UnF=\boldsymbol\eta(\F G)$. 
\end{lemma}
\begin{proof}
Note first that $\boldsymbol\eta=\hat{\boldsymbol\lambda}(1,2,3)$.
It follows that $\boldsymbol\eta\in\UnF$ and hence $\boldsymbol\eta(\F G)\subseteq\UnF$.
To prove that $\UnF\subseteq\boldsymbol\eta(\F G)$ it is enough to show that the above $\F$-basis we have constructed for $\UnF$
is contained in $\boldsymbol\eta(\F G)$.
It will be convenient in what follows to identify a permutation in the symmetric group $S_n$ with the permutation matrix in $G=\GlnF$ obtained by applying this permutation on the rows of the identity matrix.
It is then easy to check that for $g\in S_n$ and $\boldsymbol\lambda=(\lambda_{ijk})\in\F^{n^3}$, we have $\boldsymbol\lambda g=(\lambda_{ijk}')\in\F^{n^3}$, where $\lambda_{ijk}'=\lambda_{g(i),g(j),g(k)}$.
So, starting with $\boldsymbol\eta=\hat{\boldsymbol\lambda}(1,2,3)$ we can obtain (inside $\boldsymbol\eta(\F G)$) all elements of $\UnF$ of the form $\hat{\boldsymbol\lambda}(r,s,t)$, with $1\le r,s,t\le n$ and such that $r<s$ and $t\not\in\{r,s\}$, we have described above.

We next consider the action on $\boldsymbol\eta$ by the elementary matrix $g'$ obtained from the identity matrix by adding the second column of the identity matrix to its third column.
Then $\boldsymbol\eta g'=\hat{\boldsymbol\lambda}(1,2,3)g'=\hat{\boldsymbol\lambda}(1,2,3)-\hat{\boldsymbol\lambda}(1,3,2)-\tilde{\boldsymbol\lambda}(1,2,3)$, so $\tilde{\boldsymbol\lambda}(1,2,3)\in\boldsymbol\eta(\F G)$.
Considering now $\tilde{\boldsymbol\lambda}(1,2,3)$ and acting successively by the permutations $(3\,4)$, $(4\,5)$, \ldots, $(n\!-\!1\, n)$ we obtain the elements $\tilde{\boldsymbol\lambda}(1,2,4)$, $\tilde{\boldsymbol\lambda}(1,2,5)$, \ldots, $\tilde{\boldsymbol\lambda}(1,2,n)$ inside $\boldsymbol\eta(\F G)$.
On setting $g''$ to be the $n$-cycle $(n\, n\!-\!1\, \ldots\, 2\, 1)\in S_n$, we see that
$\{\tilde{\boldsymbol\lambda}(2,s,t):\ 1\le s,t\le n$ with $s\ne t$ and $2\not\in\{s,t\}\}\subseteq\boldsymbol\eta(\F G)$.
Finally, it is easy to see that further applications of $g''$ on this last subset of elements of $\boldsymbol\eta(\F G)$ generate the remaining elements of the spanning set we have constructed above.
\end{proof}

Below we discuss the composition series of $\SknF$ as an $\F G$-module.
For this we need to consider the cases $\Char\F\mbox{$\not|$}(n-1)$ and $\Char\F|(n-1)$ separately.
Recall our assumption that $\F$ is an arbitrary infinite field and that $n\ge3$.
We include a preliminary remark first.

\begin{remark}\label{RemarkRhoInMnF}
A useful observation is that $\boldsymbol\rho\in\UnF$ if, and only if, $\Char\F|(n-1)$.
This is easily seen by considering the basis $(v_1^*,\ldots,v_n^*)$ of $V$ we have fixed relative to which $\boldsymbol\rho$ is the structure vector of $\mathfrak{r}_n=(V,[,]_{\mathfrak{r}_n})\in\SknFa$ where $[v_i^*,v_n^*]=v_i^*$ for $1\le i\le n-1$ (see Remark~\ref{RemarkHeisenbergRemA}).
Then ${\rm trace}(\ad_{v_i^*})=0_{\F}$ for $1\le i\le n-1$ and ${\rm trace}(\ad_{v_n^*})=-(n-1)\cdot1_{\F}$. 

Also note that $\boldsymbol\rho(\F G)$ is an irreducible $\F G$-submodule of $\SknF$ for any field $\F$, since as sets
$\boldsymbol\rho(\F G)=O(\boldsymbol\rho)\cup\{{\bf0}\}$.
\end{remark}

(i) case $\Char\F\mbox{$\not|$}(n-1)$:
Then $\boldsymbol\rho(\F G)\cap \boldsymbol\eta(\F G)=\{{\bf0}\}$ in view of the above remark (recall $\boldsymbol\eta(\F G)=\UnF$ always).
Moreover, our assumption on $\Char\F$ ensures that $\boldsymbol\eta(\F G)$ is an irreducible $\F G$-module.
(Explanation: Suppose $M$ is a nonzero $\F G$-submodule of $\boldsymbol\eta(\F G)$.
Then there exists $\boldsymbol\lambda\in M\setminus\{{\bf0}\}$.
Since $\boldsymbol\lambda\not\in\boldsymbol\rho(\F G)$, invoking Corollary~\ref{NotPDegeneratesToMu''OtherwiseLev>1} we see that $\boldsymbol\eta\in\overline{O(\boldsymbol\lambda)}$.
But $\overline{O(\boldsymbol\lambda)}\subseteq M$ since $M$, as an $\F G$-submodule of $\F^{n^3}$, is a Zariski-closed subset of $\F^{n^3}$ containing $O(\boldsymbol\lambda)$.
It follows that $\boldsymbol\eta\in M$.
We conclude that $M=\boldsymbol\eta(\F G)$ since $\boldsymbol\eta(\F G)\subseteq M\subseteq\boldsymbol\eta(\F G)$.
This establishes that $\boldsymbol\eta(\F G)$ is irreducible when $\Char\F\not|(n-1)$.)
So, in this case, $\boldsymbol\rho(\F G)$ and $\boldsymbol\eta(\F G)$ are both irreducible $\F G$-submodules of $\SknF$.
Moreover, their dimensions as $\F$-spaces add up to $\dim_{\F}\SknF$ (see the discussion after Definition~\ref{DefUnF}) giving $\SknF=\boldsymbol\rho(\F G)\oplus\boldsymbol\eta(\F G)$.
We conclude that in this case, $\SknF$ has precisely two composition series, namely $\{{\bf0}\}\subseteq\boldsymbol\rho(\F G)\subseteq\SknF$ and $\{{\bf0}\}\subseteq\boldsymbol\eta(\F G)\subseteq\SknF$.
(This is because any irreducible $\F G$-submodule of $\SknF$ which is not equal to $\boldsymbol\rho(\F G)$ must necessarily be
$\boldsymbol\eta(\F G)$ by similar argument as above since such a module contains an element $\boldsymbol\mu$ with $\boldsymbol\eta\in\overline{O(\boldsymbol\mu)}$.)

(ii)
case $\Char\F|(n-1)$:
Then $\boldsymbol\rho\in\boldsymbol\eta(\F G)$ and hence $\boldsymbol\rho(\F G)\subseteq\boldsymbol\eta(\F G)$.
Since $\boldsymbol\rho(\F G)$ is always irreducible (regardless of the characteristic of $\F$) we get, again by similar argument as above, that
$\boldsymbol\rho(\F G)$ is the only irreducible $\F G$-submodule of $\SknF$ when $\Char\F|(n-1)$ and, in addition, that every composition series for $\SknF$ necessarily begins with $\{{\bf0}\}\subseteq\boldsymbol\rho(\F G)\subseteq\boldsymbol\eta(\F G)$.

\section{A proof of Theorem~\ref{TheoremMainPart2}}\label{SectionLevel1NonSkewSym}

We continue with our hypothesis that $\F$ is an arbitrary infinite field. Also let $n\ge2$ for this section.

\begin{definition}\label{DefRemainingLev1Alg}
Following~\cite{Khudoyberdiyev&Omirov2013} we introduce the algebra structures~$\mathfrak{d}_n$ and~$\mathfrak{e}_n(\alpha)$, for $\alpha\in\F$, as follows.
Let $\boldsymbol{\delta}_n=(\delta_{ijk})\in\F^{n^3}$ be the structure vector which has $\delta_{112}=1_{\F}$ as its only nonzero component.
Also, for $\alpha\in\F$, let $\boldsymbol{\varepsilon}_n(\alpha)=(\varepsilon_{ijk}(\alpha))\in\F^{n^3}$ be the structure vector which has $\varepsilon_{111}(\alpha)=1_{\F}$, $\varepsilon_{1ii}(\alpha)=\alpha$ (for $2\le i\le n$) and $\varepsilon_{i1i}(\alpha)=(1_{\F}-\alpha)$ (for $2\le i\le n$) as its only components which can possibly be nonzero.
Finally define the algebra structures $\mathfrak{d}_n,\mathfrak{e}_n(\alpha)\in\AnF$ by $\mathfrak{d}_n=\Theta^{-1}(\boldsymbol{\delta}_n)$ and $\mathfrak{e}_n(\alpha)=\Theta^{-1}(\boldsymbol{\varepsilon}_n(\alpha))$ with $\Theta$ as in Remark~\ref{RemarkAnFFn3AreVectorSpaces}.
\end{definition}

We now collect some immediate consequences of this definition.

\begin{remark}\label{RemarkPropOfDE}
(i) Fix $\alpha\in\F$ and let $(\beta_1,\ldots,\beta_n)\in\F^n$.
Also let $\mathfrak{g}(\alpha;\beta_1,\ldots,\beta_n)=(V,[,])\in\AnF$ be the algebra structure which, relative to our basis $(v_i^*)_{i=1}^n$, has the following as the only components which can possibly be nonzero: $[v_i^*, v_i^*]=\beta_i$, ($1\le i\le n$) and $[v_i^*,v_j^*]=\alpha\beta_iv_j^*+(1-\alpha)\beta_jv_i^*$ (for $1\le i,j\le n$, $i\ne j$).
It is easy to check that $O(\boldsymbol{\varepsilon}_n(\alpha))=\{\Theta(\mathfrak{g}(\alpha;\beta_1,\ldots,\beta_n)):$ $(\beta_1,\ldots,\beta_n)\ne(0_{\F},\ldots,0_{\F})\}$.
It follows that, for each $\alpha\in\F$, the union $O(\boldsymbol{\varepsilon}_n(\alpha))\cup\{{\bf0}\}$ is an $\F$-subspace of $\F^{n^3}$, and $\overline{O(\boldsymbol{\varepsilon}_n(\alpha))}=O(\boldsymbol{\varepsilon}_n(\alpha))\cup\{{\bf0}\}$.
Moreover, for $\alpha_1$, $\alpha_2\in\F$ with $\alpha_1\ne\alpha_2$ we have $O(\boldsymbol{\varepsilon}_n(\alpha_1))\cap O(\boldsymbol{\varepsilon}_n(\alpha_2))=\varnothing$.
Finally note that $\boldsymbol{\delta}_n$ does not belong to any of these orbits.

(ii) The algebra structure $\mathfrak{d}_n=(V,[,]_{\mathfrak{d}_n})\in\AnF$ satisfies the commutativity relation $[x,y]_{\mathfrak{d}_n}=[y,x]_{\mathfrak{d}_n}$ for all $x,y\in V$.
Moreover, $\mathfrak{d}_n\in\BnFa$ and $\dim_{\F}(\ann\mathfrak{d}_n)=\dim_{\F}(\ann_L\mathfrak{d}_n)=n-1$.
\end{remark}

\begin{lemma}\label{LemmaDnIsLevel1}
$\overline{O(\boldsymbol{\delta}_n)}=O(\boldsymbol{\delta}_n)\cup\{{\bf0}\}$.
\end{lemma}

\begin{proof}
Let $\mathfrak{g}=(V,[,]_{\mathfrak{g}})\in\AnF$ be a proper degeneration of $\mathfrak{d}_n$ with $\mathfrak{g}\ne\mathfrak{a}_n$.
In view of the above remark we get that $\mathfrak{g}\in\BnFa$ and also that $[x,y]_{\mathfrak{g}}=[y,x]_{\mathfrak{g}}$ for all $x,y\in V$.
Invoking Lemma~\ref{LemmaDimAnnIncreasesDimSquareDecreases} we get in addition that $\dim_{\F}(\ann\mathfrak{g})=\dim_{\F}(\ann_L\mathfrak{g})=n-1$.
We can thus complete a basis $(e_i)_{i=2}^n$ of $\ann\mathfrak{g}$ to a basis $(e_i)_{i=1}^n$ of $V$.
Then $[e_1,e_1]_{\mathfrak{g}}\ne0_V$ (otherwise we would have $\ann\mathfrak{g}=V$).
The fact that $\mathfrak{g}\in\BnFa$ also ensures that $[e_1,e_1]_{\mathfrak{g}}\in\ann\mathfrak{g}$ and $[e_1,e_1]_{\mathfrak{g}}\not\in\Fspan(e_1)$.
Hence we can consider a basis $(v_i)_{i=2}^n$ of $\ann\mathfrak{g}$ with $v_2=[e_1,e_1]_{\mathfrak{g}}$.
Completing to a basis $(v_i)_{i=1}^n$ of $V$ with $v_1=e_1$ we get that $\mathfrak{g}\cong \mathfrak{d}_n$.
\end{proof}

For the remaining lemmas we use exactly the same ideas and line of proof as in~\cite[Proposition~2.2]{Khudoyberdiyev&Omirov2013}.
To obtain results over an arbitrary field we make use of Lemma~\ref{LemmaTechnique*}.

\begin{lemma}\label{(*)DegeneratesToDn}
Let $\mathfrak{g}=(V,[,])\in\AnF$ and suppose $\mathfrak{g}$ does not satisfy condition~$(**)$.
Then $\mathfrak{g}$ degenerates to $\mathfrak{d}_n$.
\end{lemma}
\begin{proof}
Assume the hypothesis and let $x\in V$ be such that $[x,x]\not\in\Fspan(x)$.
Clearly $x\ne0_V$ and we can consider a basis $(e_i)_{i=1}^n$ of $V$ with $e_1=x$ and $e_2=[x,x]$.
Let $\boldsymbol\lambda=(\lambda_{ijk})$ be the structure vector of $\mathfrak{g}$ relative to this basis and set $\hat q=(q_i)_{i=1}^n\in\mathbb Z^n$ with $q_1=1$ and $q_i=2$ for $i\ge2$.
Invoking Lemma~\ref{LemmaTechnique*} we get the desired result.
\end{proof}

Combining the results of this section so far with Lemma~\ref{LemmaNot(*)DegeneratesToHeis} and Corollary~\ref{CorolPsatisfiesCond*}, in order to complete the proof of Theorem~\ref{TheoremMainPart2} it suffices to consider the case $\mathfrak{g}\in\AnF\setminus\SknFa$ with $\mathfrak{g}$  satisfying condition~$(*)$:

\begin{lemma}\label{}
Let $\mathfrak{g}=(V,[,])\in\AnF\setminus\SknFa$ and suppose $\mathfrak{g}$ satisfies condition~$(*)$.
Then $\mathfrak{g}$ degenerates to $\mathfrak{e}_n(\alpha)$ for some $\alpha\in\F$.
\end{lemma}
\begin{proof}
Assume the hypothesis.
Since there exists $x\in V$ with $[x,x]\ne0_V$, we can consider, by scaling if necessary, a basis $(e_i)_{i=1}^n$ of $V$ such that $[e_i,e_i]=e_i$ for $1\le k\le n$ and $[e_i,e_i]=0$ for $k+1\le i\le n$ (for some $k\ge 1$).
Now let $\xi\in\F^*$ and fix $i$ with $2\le i\le k$.
Considering the product $[e_1+\xi e_i, e_1+\xi e_i]$ (which equals to $\alpha_\xi(e_1+\xi e_i)$ for some $\alpha_\xi\in\F$) and combining with the fact that $[e_1,e_i]+[e_i,e_1]=\gamma_1e_1+\gamma_ie_i$ for some (constant) $\gamma_1,\gamma_i\in\F$ we get that $\alpha_\xi=\gamma_i+\xi=1_{\F}+\xi\gamma_1$.
As a consequence we have that $(1_{\F}-\gamma_i)+(\gamma_1-1_{\F})\xi=0_{\F}$.
This is true for all $\xi\in\F^*$ from which we deduce that $\gamma_1=\gamma_i=1_{\F}$.
We conclude that $[e_1,e_i]+[e_i,e_1]=e_1+e_i$ for $2\le i\le k$.
Moreover, setting $\xi=-1_{\F}$ in the above argument we get that $[e_1-e_i,e_1-e_i]=0_{\F}$ for $2\le i\le k$.

Next we consider the basis $(e_i')_{i=1}^n$ of $V$ where $e_1'=e_1$, $e_i'=e_1-e_i$ for $2\le i\le k$ and $e_i'=e_i$ for $k+1\le i\le n$.
By similar argument as above (considering the product $[e_1'+\xi e_i', e_1'+\xi e_i']$ for $2\le i\le n$) we get that $[e_1',e_i']+[e_i',e_1']=e_i'$ for $2\le i\le n$.

Now let $\boldsymbol{\lambda}=(\lambda_{ijk})\in\F^{n^3}$ be the structure vector of $\mathfrak{g}$ relative to the basis $(e_i')$ of $V$.
In view of the above, we have that $\lambda_{111}=1_{\F}$ and $\lambda_{i1i}+\lambda_{1ii}=1_{\F}$ for $2\le i\le n$.
But the $\lambda_{i1i}$ all have the same value for $2\le i\le n$
(compare with the proof of Corollary~\ref{CorolPsatisfiesCond*} --- we use the fact that $\mathfrak{g}$ satisfies condition~$(*)$).
We can thus complete the proof by setting $\hat q=(q_i)_{i=1}^n\in\mathbb Z^n$ where $q_1=0$ and $q_i=1$ for $2\le i\le n$ in Lemma~\ref{LemmaTechnique*}.
\end{proof}

Note that the above arguments ensure that for $n=2$ the algebras $\mathfrak{d}_2$ and $\mathfrak{e}_2(\alpha)$ for $\alpha\in\F$ give a complete list of elements of $\mathop{\mathcal A_2(\F)}\setminus \mathop{\mathcal{K}_2^a(\F)}$ which have $\mathfrak{a}_2$ as their only proper degeneration.

\subsection*{Acknowledgments}
We would like to thank Professor A.E. Zalesski for useful discussions on this subject.


\begin{thebibliography}{hhhhhh}\itemsep=-1ex
\footnotesize



\bibitem{Borel}
Borel A., {\it Linear Algebraic Groups}, Springer-Verlag, New York, 1991.

\bibitem{Burde&Steinhoff1999}
Burde~D., Steinhoff~C., Classification of orbit closures of 4-dimensional complex Lie algebras,
{\it Journal of Algebra} {\bf214} (1999), 729--739.

\bibitem{Curtis1984}
C. Curtis,
{\it Linear Algebra: An Introductory Approach (Undergraduate Texts in Mathematics)}, 4th Edition,
Springer-Verlag New York, 1984.


\bibitem{FialowskyOHalloran1990}
A. Fialowsky and J. O'Halloran,
A comparison of deformaions and orbit closure,
{\it Comm. in Algebra} {\bf18} (1990), 4121--4140.

\bibitem{Geck2003}
M. Geck,
{\it An Introduction to Algebraic Geometry and Algebraic Groups}, Oxford University Press, 2003.

\bibitem{Gerstenhaber1964}
M. Gerstenhaber,
On the deformations of rings and algebras,
{Ann. of Math.} {\bf79} (1964), 59--103.

\bibitem{Gorbatsevich1991}
V.V. Gorbatsevich, Contractions and degenerations of finite-dimensional algebras,
{\it Izv. Vyssh. Uchebn. Zaved., Mat.} {\bf10} (1991), 19--27 (Russian); translated in {\it Soviet Math. (Iz. VUZ)} {\bf35} (1991), 17--24.

\bibitem{Gorbatsevich1993}
V.V. Gorbatsevich, Anticommutative finite-dimensional algebras of the first three levels of complexity,
{\it Algebra Anal.} {\bf5} (1993), 100--118 (in Russian); translated to English {\it St. Petersbg. Math. J.} {\bf5} (1994), 505--521.

\bibitem{Gorbatsevich1998}
V.V. Gorbatsevich, On the level of some solvable Lie algebras, {\it Sibirskiy Matematicheskiy Zhurnal} {\bf39} (1998), 1013--1027 (in Russian);
translated to English {\it Sib. Math. J.} {\bf39} (1998), 872--883.

\bibitem{GrunewaldOHalloran1988a}
F. Grunewald and J. O'Halloran,
Varieties of nilpotent Lie algebras of dimensions less then six.
{\it J. Algebra} {\bf112} (1988), 315--325.

\bibitem{GrunewaldOHalloran1988b}
F. Grunewald and J. O'Halloran,
A characterization of orbit closure and applications
{\it J. Algebra} {\bf116} (1988), 163--175.

\bibitem{Humphrays1991}
J. Humphrays, {\it Linear algebraic groups},
Linear Algebraic Groups, Graduate Texts in Mathematics 21, Berlin, New York: Springer-Verlag, 1998 (5th edition).

\bibitem{InonuWigner1953}
E. In\"on\"u and E.P. Wigner,
On the contraction of groups and their representations,
{\it Proc. Nat. Acad. Sci. U.S.A.} {\bf39} (1953), 510--524.

\bibitem{InonuWigner1954}
E. In\"on\"u and E.P. Wigner,
On a particular type of convergence to a singular matrix,
{\it Proc. Nat. Acad. Sci. U.S.A.} {\bf40} (1954) 119--121.


\bibitem{Jacobson1962}
N. Jacobson, {\it Lie Algebras,} Interscience, New York, 1962.

\bibitem{Kaygorodov&Popov&Volkov2016}
I. Kaygorodov, Yu. Popov and Yu. Volkov,
Degenerations of binary Lie and nilpotent Malcev algebras,
arXiv:1609.07392.

\bibitem{Kaygorodov&Volkov2017}
I. Kaygorodov and Yu. Volkov,
The variety of 2-dimensiomal algebras over an algebraically closed field,
arXiv:1701.08233.

\bibitem{Khudoyberdiyev&Omirov2013}
 A.Kh.Khudoyberdiyev and B.A. Omirov,
The classification of algebras of level one,
{\it Linear Algebra and its Applications} {\bf439} (2013), 3460--3463.

\bibitem{Levy-Nahas1967}
M. Levy-Nahas, Deformation and contraction of Lie algebras,
{\it J. Math. Phys.} {\bf8} (1967), 1211--1223.

\bibitem{Lauret2003}
Lauret J., Degenerations of Lie algebras and Geometry of Lie groups,
{\it Differ. Geom. Appl.} {\bf18} (2003), 177--194.

\bibitem{NesterenkoPopovych2006}
M. Nesterenko and R. Popovych,
Contractions of low-dimensional Lie algebras,
{\it J. Math. Phys.} {\bf47} (2006), 123515.

\bibitem{Popov2009}
V.L. Popov, Two orbits: When is one in the closure of the other?,
{\it Proc. Steklov Math. Inst.} {\bf264} (2009), 146--158.

\bibitem{Rakhimov2005}
I.S. Rakhimov,
On the degenerations of finite dimensional nilpotent complex Leibniz algebras,
{\it Zapiski Nauch. Semin. POMI} {\bf319} (2005), 1--7.

\bibitem{Seeley1990}
C. Seeley, Degenerations of 6-dimensional nilpotent Lie algebras over $\mathbb{C}$,
{\it Comm. in Algebra} {\bf18(10)} (1990), 3493--3503.

\bibitem{Seeley1991}
C. Seeley, Degenerations of central quotients,
{\it Arch. Math.} {\bf56} (1991), 236--241.

\bibitem{Segal1951}
I.E. Segal,
A class of operator algebras determined by groups,
{\it Duke Math J.}, {\bf18} (1951), 221--265.



\end{thebibliography}
\end{document}